\documentclass[preprint]{elsarticle}
\usepackage{graphicx} % support the \includegraphics command and options
\usepackage{pdfpages} % for input a page which is a pdf file
\usepackage{tikz} % for draw a picture
\usepackage{amsmath,amssymb,amsthm,amscd} % for maths
\usepackage{mathtools}
\usepackage{bm} % for using \bm, which describes a vector (with bold font)
\usepackage{pgfplots} % for using gnuplot in pdfLaTeX
\usepackage{array} % for better arrays (eg matrices) in maths
\usepackage{paralist} % very flexible & customisable lists (eg. enumerate/itemize, etc.)
\usepackage{subfig} % make it possible to include more than one captioned figure/table in a single float
\usepackage{comment} % for using \begin{comment} ... \end{comment}
\usepackage{standalone} % for using documentclass{standalone}
\usepackage{import} % for include pictures which are in each directly
\usepackage{currfile} % for include data files for pgfplots
\usepackage{lineno,hyperref}
\usepackage{amsmath,amssymb}
\usepackage{mathtools}
\usepackage{amsthm}
\usepackage{bm}
\usepackage[all]{xy}

\usepackage{xcolor}
\newcommand{\dual}[1]{\langle{#1}\rangle}
\newcommand{\Div}{\operatorname{div}}
\newcommand{\Rot}{\operatorname{rot}}
\newcommand{\Grad}{\operatorname{grad}}

\newcommand{\ed}{\operatorname{d}}
\newcommand{\eD}{\operatorname{D}}
\newcommand{\eL}{\operatorname{L}}
\newcommand{\Vol}{\mathsf{vol}}

\newtheorem{theorem}{Theorem}

\newtheorem{proposition}{Proposition}

\modulolinenumbers[5]

\journal{Jounal of \LaTeX\ Templates}

\begin{document}
\begin{frontmatter}

\title{A finite element method to a periodic
steady-state problem for an electromagnetic field system using the space-time finite element exterior calculus}

\author[UT]{Masaru \textsc{Miyashita}\corref{CO}}
\cortext[CO]{Corresponding author}
\ead{miyashita-masaru537@g.ecc.u-tokyo.ac.jp}

\author[UT]{Norikazu \textsc{Saito}}
\address[UT]{Graduate School of Mathematical Sciences, The University of Tokyo,
Komaba 3-8-1, Meguro-ku, Tokyo 153-8914, Japan}




\begin{abstract}
This paper proposes a finite element method for solving the periodic steady-state problem for the scalar-valued and vector-valued Poisson equations, a simple reduction model of the Maxwell equations under the Coulomb gauge. Introducing a new potential variable, we reformulate two systems composed of the scalar-valued and vector-valued Poisson problems to a single Hodge-Laplace problem for the $1$-form in $\mathbb{R}^4$ using the standard de Rham complex. Consequently, we can apply the Finite Element Exterior Calculus (FEEC) theory in $\mathbb{R}^4$ directly to deduce the well-posedness, stability, and convergence. Numerical examples using the cubical element are reported to validate the theoretical results.
\end{abstract}

\begin{keyword}
Finite Element Exterior Calculus \sep Maxwell equation \sep periodic steady state analysis \sep Hodge Laplacian \sep Cubical element
\MSC[2020]
65N12\sep %Stability and convergence of numerical methods
%for boundary value problems involving PDEs
65N30\sep % Finite element, Rayleigh-Ritz and Galerkin
%methods for boundary value problems involving
%PDEs
35J25% Boundary value problems for second-order elliptic
%equations
\end{keyword}

\end{frontmatter}
\section{Introduction}
\label{sec:intro}
Let $\Omega$ be a bounded Lipschitz domain in $\mathbb{R}^3$.
We consider the space-time region $Q:=(0,T) \times \Omega$ with a given $T>0$, and set the lateral boundary $\partial Q= (0,T) \times \partial \Omega$.
The target problem in this paper is the following coupling problem composed of the scalar-valued Poisson equation, vector-valued Poisson equation, and divergence-free constraint with the essential boundary conditions (see \cite{Jackson}):
\begin{subequations} % 2021-10-15 09:10の式群
    \label{mainVA}
  \begin{align}
    -\Div \Grad\phi &= \bar{\rho} && \mbox{ in }Q, \label{vecLap}\\
    \phi &= 0 &&\mbox{ on } \partial Q, \label{mainVAb}\\
    \Rot\Rot A &= \bar{j} &&\mbox{ in }  Q,\label{mainVAc}\\
    \Div A &= 0 &&\mbox{ in } Q,\label{mainVAd}\\
    n \times A &= 0 && \mbox{ on }  \partial Q.\label{mainVAe}
  \end{align}
\end{subequations}
Therein, a scalar-valued function $\phi$ of $t\in [0,T]$ and $x\in\overline{\Omega}$ denotes the scalar potential, and a vector-valued function $A$ of $t\in [0,T]$ and $x\in\overline{\Omega}$ denotes the vector potential.
Assume that the charge density $\bar{\rho}$ and current density $\bar{j}$, respectively, are given scalar-valued and vector-valued functions. Moreover, assume that $\bar{\rho}$ and $\bar{j}$ are continuous and $T$-periodic with respect to the time variable. That is, $\bar{\rho}(x,0) = \bar{\rho}(x,T)$ and $\bar{j}(x,0) = \bar{j}(x,T)$ for all $x \in \Omega$. The outer unit normal vector to $\partial M$ is denoted by $n$.
Here and hereinafter, we use the standard notation of the vector calculus.
It should be noticed that the function $A$ satisfies
\[
  \Rot\Rot A- \Grad\Div A=\bar{j}\mbox{ in }Q,\qquad
n\times A=0\mbox{ on }\partial Q. \qquad (\star)
\]
Therefore, we call \eqref{mainVAc}, \eqref{mainVAd} and \eqref{mainVAe} the vectoe-valued Poission equation with the divergence-free constraint.

Many components and equipments such as motors, engines, turbines,
and plasma-based etching and deposition systems are operated in periodic steady-state conditions \cite{Jackson},\cite{Plasma}.
The system \eqref{mainVA} is a simple reduction model of the Maxwell equations under the Coulomb gauge in a periodic steady-state. Therein, the gradient of the scalar potential $\phi$ denotes the static electric field, and the time derivative of the vector potential $A$ denotes the inductively electric field. The performance and lifetime of the plasma source are estimated by the induced electric field and electrostatic field, respectively. Therefore, it is important and challenging for calculating the periodic steady-state described by (\eqref{mainVA}) in the plasma equipment simulation.

Although equations \eqref{vecLap} and \eqref{mainVAc} are just linear partial equations, their numerical computations have some difficuluties.
Problem \eqref{vecLap} {and} \eqref{mainVAb} {is} the scalar-valued Poisson equation with the homogeneous Direret boundary condition at each time. In the mixed finite element method, {we} introduce an intermediate variable$E = \Grad \phi$.
{The unkowns} $\phi$ and $E$ are solved{, respectively, } in the $r$-th order polynomial space $\mathcal{P}_r(\Omega)$ and the $(r-1)$ th order {one} $\mathcal{P}_{r-1}(\Omega)^3$ as a vector-valued function.
The combination of $\mathcal{P}_{r}(\Omega)$ and $\mathcal{P}_{r-1}(\Omega)^3$ is known to cause issues such as numerical oscillation {(see \cite{bbf2013})}.
{On theother hand, }Equations \eqref{mainVAc}--\eqref{mainVAe} often {appear} in static magnetic fields problem at each time.
Equation \eqref{mainVAd} is called the Coulomb condition and is solved simultaneously with vector-valued Poisson equation {($\star$)} in a mixed formulation using the Lagrange multiplier method. The condition of \eqref{mainVAe} is called the metal boundary condition and is {a kind of the essential boundary condition} Each component of $A$ is related each other due to conservation law $\Div A$.

The theory of the \emph{Finite Element Exterior Calculus} (FEEC) gives
a useful framework to solve these problems. Actually, the scalar-valued and vector-valued Poisson problems are formulated by the Hodge Laplacian problem for the $0$-form and $1$-form, respectively, using the standard de Rham complex in $\mathbb{R}^3$ (see \cite{Tu} for example).
Applying the FEEC theory, we can derive a mixed weak formulation and constract stable finite element spaces in a coherent manner; see \cite{Arnold(2018)} \cite{afw06} \cite{afw10}. That is, the scaler-valued Poisson problem \eqref{vecLap}--\eqref{mainVAb} and vector-valued one \eqref{mainVAc}--\eqref{mainVAe} are solved by the stable finite element method \emph{separately}. We review this point in Section 2.
The purpose of this paper is to propose alternate (and somewhat new) approach. We consider $Q$ as a subset of $\mathbb{R}^4=\mathbb{R}^{1+3}$ and formulate \eqref{vecLap}-\eqref{mainVAe} as a boundary value problem for the Hodge-Laplacian problem on the 4-dimensional space-time region $Q$ with a $4$-dimensional potential as an unknown variable.
To be more specific, we introduce a new potential $u$ as a direct product $(\phi,A)$ and express it as
\[
 u=dt\wedge \phi+A
\]
using the dual basis $dt$ of the canonical basis corresponding to the time variable and the wedge product $\wedge$.
If we interpret $\phi$ and $A$ as the $0$-form and $1$-form, the new potential $u$ becomes the $1$-form in $\mathbb{R}^4$.
Moreover, two systems composed of the scalar-valued and vector-valued Poisson problems imply a single Hodge--Laplace problem for the $1$-form in $\mathbb{R}^4$.This is possible because the problem \eqref{mainVA} does not contain the time derivative term. Consequently, we can apply the FEEC theory in $\mathbb{R}^4$ directly to deduce the well-posedness, stability and convergence.
Of course, it is in general difficult to find a suitable finite element space in $\mathbb{R}^4$. We restrict our consideration to the cubical element that is a poroduct of the interval element. This enable us to extend the results for $\mathbb{R}^3$ to those for $\mathbb{R}^4$.

The FEEC theory has contributed to the development of higher-order Whitney elements \cite[Chapter 7]{Arnold(2018)}. Furthermore, FEEC is considered as a unified theory of finite element methods and one of the theoretical bases for the development of structure-preserving schemes in more complex problems \cite{structure}. For example, FEEC {gives} the structure preserving scheme in the calculation of electromagnetic field on the Vlasov-Maxwell system \cite{Kraus}. Here, another approximation theory of differential form, Discrete Exterior Calculus (DEC) \cite{DEC1} is also used for the calculation of the Vlasov-Maxwell system \cite{squire2012}. As an attempt to include the time axis, {we know} Salamon's work of Space Time FEEC \cite{Salamon}. Quenneville--Belair dealt with the time evolution problem of Maxwell equations in the 3-dimensional FEEC \cite{Quennevill}.

{This study reports new applications of the FEEC theory. }
Our novel feature{ is to utilize }a mesh in {the} 4-dimensional space-time and {solve} the Hodge Laplacian problem on the 4-dimensional periodic steady condition using the FEEC framework.

This paper is organized as follows.
We review the FEEC theory of $\mathbb{R}^3$ in \S 2. In \S 3, we derive our proposed Hilbert complex {and formulate} the main problem as Hodge--Laplacian problem. {Then, we discuss} the well{-}posedness. In \S 4, {we state} {a} numerical simulation scheme of the Hodge laplacian problem as saddle point problem. \S 5 shows the numerical examples for support of our theoretical discussion

\section*{Notation}
We use the standard Lebesgue
$L^2(\Omega,\mathbb{R}^d)$ for $d=1,\ldots,4$ and set
$L^2(\Omega)=L^2(\Omega,\mathbb{R})$. The standard Sobolev spaces are also used:
\begin{subequations} % 2021-10-15 10:46の式群
\label{eq:fspace}
 \begin{align}
H^1(\Omega) &=\{v\in L^2(\Omega) \mid \nabla v\in L^2(\Omega;\mathbb{R}^3)\},\label{eq:fspace1}\\
H(\Div) &=\{A\in L^2(\Omega;\mathbb{R}^3) \mid \Div A\in L^2(\Omega)\},\label{eq:fspace3}\\
H(\Rot) &=\{A\in L^2(\Omega;\mathbb{R}^3) \mid \Rot A\in L^2(\Omega;\mathbb{R}^3)\},\label{eq:fspace4}\\
\mathring{H}^1(\Omega) &=\{v\in H^1(\Omega) \mid v=0\mbox{ on }\partial\Omega\},\label{eq:fspace2}\\
\mathring{H}(\Div) &=\{A\in H(\Div) \mid A\cdot n=0\mbox{ on }\partial\Omega\},\label{eq:fspace6}\\
\mathring{H}(\Rot) &=\{A\in H(\Rot) \mid A\times n=0\mbox{ on }\partial\Omega\}.\label{eq:fspace7}
\end{align}
\end{subequations}

%%%
%%% Sec: FFEC R^3

\section{Brief review of the FEEC in $\mathbb{R}^3$}
\label{sec:FEEC-R3}

Before studying the main target problem \eqref{mainVA}, we review the FEEC theory using the steady-state version of \eqref{mainVA} :
\begin{subequations} % 2021-10-15 09:10の式群
    \label{mainVA-s}
  \begin{align}
    -\Div \Grad\phi &= \bar{\rho} && \mbox{ in }\Omega, \label{vecLap-s}\\
    \phi &= 0 &&\mbox{ on } \partial \Omega, \label{mainVAb-s}\\
    \Rot\Rot A &= \bar{j} &&\mbox{ in }  \Omega,\label{mainVAc-s}\\
    \Div A &= 0 &&\mbox{ in } \Omega,\label{mainVAd-s}\\
    n \times A &= 0 && \mbox{ on }  \partial \Omega.\label{mainVAe-s}
  \end{align}
\end{subequations}
All functions in this section are supposed to be time-independent. We use the symbols $\phi$ and $A$ in the periodic steady-state problem \eqref{mainVA} and steady-state problem \eqref{mainVA-s}, since there is no fear of confusion.
This section is based on \cite[Chapters 4 and 5]{Arnold(2018)}.
In order to state the reformulation of \eqref{mainVA-s} in terms of the exterior calculus, we first recall a suitable Hilbert complex. It should be noticed that the $L^2$ de Rham complex with no boundary condision is disccussed in \cite{Arnold(2018)}. In particular, we work on the $L^2$ de Rham complex associated with $\Omega$ with boundary conditions. The base Hilbert spaces are $W^0=W^3=L^2(\Omega)$, $W^1=W^2=L^2(\Omega;\mathbb{R}^3)$. The operators are defined as $\ed^0=\Grad$, $\ed^1=\Rot$, and $\ed^2=\Div$ with domains $V^0=\mathring{H}^1(\Omega)$, $V^1=\mathring{H}(\Rot)$, $V^2=\mathring{H}(\Div)$ and $V^3={L}^2(\Omega)$, respectively. The domain complex is described as
\begin{subequations} % 2021-10-15 13:03の式群
 \label{eq:dc-r3-1}
 \begin{equation}
 \label{eq:dc-r3-1a}
\xymatrix{
  0 \ar[r] & V^0 \ar[r]^-{\ed^0}
  & V^1 \ar[r]^-{\ed^1}
  & V^2 \ar[r]^-{\ed^2}
  & V^3 \ar[r] & 0,
}
\end{equation}
or, equivalently,
\begin{equation}
 \label{eq:dc-r3-1b}
\xymatrix{
  0 \ar[r] & \mathring{H}^1(\Omega) \ar[r]^-\Grad
  & \mathring{H}(\Rot) \ar[r]^-{\Rot}
  & \mathring{H}(\Div) \ar[r]^-\Div
  & {L}^2(\Omega) \ar[r] & 0.
}
\end{equation}
\end{subequations}
The dual complex is given as
\begin{subequations} % 2021-10-15 13:03の式群
 \label{eq:dc-r3-2}
 \begin{equation}
 \label{eq:dc-r3-2a}
\xymatrix{
  0 & \ar[l] V_0^*
  & \ar[l]_-{\ed_1^*} V_1^*
  &\ar[l]_-{\ed_2^*} V_2^*
  & \ar[l]_-{\ed_3^*} V_3^*
  & \ar[l] 0
}
\end{equation}
or, equivalently,
\begin{equation}
 \label{eq:dc-r3-2b}
\xymatrix{
  0 & \ar[l] {L}^2(\Omega)
  & \ar[l]_-{-\Div} {H}(\Div)
  &\ar[l]_-{\Rot} {H}(\Rot)
  & \ar[l]_-{-\Grad} {H}^1(\Omega)
  & \ar[l] 0,
}
\end{equation}
\end{subequations}
where we have set $\ed_1^*=-\Div$, $\ed_2^*=\Rot$, and $\ed_3^*=-\Grad$ with domains $V_0^*=L^2(\Omega)$, $V_1^*={H}(\Div)$, $V_2^*={H}(\Rot)$ and $V_3^*={H}^1(\Omega)$.

All $\ed^k$ and $\ed_k^*$ are closed densely defined linear operators. Moreover, we have $\ed^{k+1}\ed^k=0$ and $\ed_{k}^*\ed_{k+1}^*=0$. That is, we have
$\mathcal{R}(\ed^{k})\subset \mathcal{N}(\ed^{k+1})$ and
$\mathcal{R}(\ed_{k+1}^*)\subset \mathcal{N}(\ed_{k}^*)$. These imply that \eqref{eq:dc-r3-1} and \eqref{eq:dc-r3-2} are Hilbert complexes (see \cite[Definition 4.1]{Arnold(2018)}). Furthermore, $\ed_{k+1}^*$ is the adjoint operator of $\ed^k$.

The $L^2$ de Rham complex has the following property (\cite[p. 38]{Arnold(2018)}).

\begin{proposition}
\label{pr:compact-r3}
$V^k\cap V_k^*$ is compactly included in $W^k$ for $k=0,1,2,3$.
\end{proposition}
%We recall that a pair of a sequence of Hilbert spaces $W^k$ and a sequence linear operator $d^k:W^k\to W^{k+1}$ satisfying  is called a Hilbert complex.

%The FEEC theory deals with the Hodge-Laplacian problem on the Hilbert complex, Hilbert spaces connected by a closed densely operator.
%We check that the Hodge-Laplacian problem includes the well-known scalar Laplacian problem and the vector Laplacian problem in this section.
%FEEC theory is a unified framework for discussing the mixed weak formulation of the Hodge-Laplacian problem, including well-posedness and error evaluation.

Set
\begin{equation}
 \label{eq:r3-func1}
\mathfrak{B}^k=\mathcal{R}(\ed^{k-1}),\quad
\mathfrak{Z}^k=\mathcal{N}(\ed^{k}),\quad
\mathfrak{B}_k^*=\mathcal{R}(\ed_{k+1}^*),\quad
\mathfrak{Z}_k^*=\mathcal{N}(\ed_{k}^*).
\end{equation}

An element $v\in W^k$ is called a harmonic $k$-form, if $\ed^kv=0$ and $\ed_{k}^*v=0$. The set of all harmonic $k$-form is denoted by $\mathfrak{H}^k$. We know that $\mathfrak{H}^k=\mathfrak{Z}^k\cap\mathfrak{Z}_k^*=\mathfrak{Z}^k\cap \mathfrak{B}^{k,\bot}$. It can be verified that, if $\Omega$ is simply-connected,
\begin{subequations}
  \label{eq:hh}
  \begin{align}
    \mathfrak{H}^0&=\{v\in \mathring{H}^1(\Omega)\mid \Grad v=0\}=\{0\}, \label{eq:hh0}\\
    \mathfrak{H}^1&=\{v\in =\mathring{H}(\Rot)\cap H(\Div)\mid \Rot A=-\Div A=0\}=\{0 \}.\label{eq:hh1}
  \end{align}
\end{subequations}

Moreover, we have (\cite[Theorems 4.5 and 4.6]{Arnold(2018)})) the following.

\begin{proposition}[Hodg decomposition]
\label{pr:h-decomp-r3}
We have the orthogonal decomposition
$W^k=\bar{\mathfrak{B}}^k\bigoplus \mathfrak{H}^k\bigoplus \bar{\mathfrak{B}}_k^*$ and
$V^k=\bar{\mathfrak{B}}^k\bigoplus \mathfrak{H}^k\bigoplus \mathfrak{Z}^{k\bot_V}$, where $\mathfrak{Z}^{k\bot_V}=\bar{\mathfrak{B}}^*_k\cap V_k$.
\end{proposition}

\begin{proposition}[Poincar{\' e} inequality]
\label{pr:h-poin-r3}
There exists a positive constant $c_P$ such that $\|u\|_{V^k}\le c_P\|\ed^ku\|$ for $u\in \mathfrak{Z}^{k\bot_V}$.
\end{proposition}

At this stage, the Hodge Laplacian $\Delta^k : W^k \rightarrow W^k$ is defined as
\begin{subequations}
\label{eq:hl-r3-1}
\begin{equation}
\label{eq:hl-r3-1a}
\Delta^k = \ed^{k-1}\ed_k^* + \ed_{k+1}^* \ed^k
\end{equation}
with its domain
\begin{equation}
\label{eq:hl-r3-1b}
D(\Delta^k) = \{ u\in V^k \cap V_k^* \mid  \ed^k u \in V_{k+1}^*,~ \ed_k^* u\in V^{k-1}\}.
\end{equation}
\end{subequations}

The Hodge Laplace problem in a strong form is described as follows:
Given $f\in W^k$, find $u\in D(\Delta^k)$ such that
 \begin{equation}
\label{eq:hlp-r3-1}
\Delta^ku=f-P_{\mathfrak{H}}f,\qquad
u\bot \mathfrak{H}^k,
% = d^{k-1}d_k^* + d_{k+1}^* d^k
\end{equation}
where $P_\mathfrak{H}$ denotes the orthogonal projection form $W^k$ onto $\mathfrak{H}^k$.

On the other hand, the Hodge Laplace problem in a primal weak form is:
Given $f\in W^k$, find $u\in V^k\cap V_k^*$ such that $u\bot \mathfrak{H}^k$ and
\begin{equation}
\label{eq:hlp-r3-2}
\dual{\ed^ku,\ed^kv}+\dual{\ed^*_ku,\ed^*_kv}=\dual{f-P_{\mathfrak{H}}f,v}\quad (\forall v \in V^k\cap V_k^*).
\end{equation}

Finally, the Hodge Laplace problem in a mixed weak form is:
Given $f\in W^k$, find $\sigma\in V^{k-1}$,
$u\in V^k$ and $p\in \mathfrak{H}^k$ such that %$u\bot \mathfrak{H}^k$ and
\begin{subequations}
\label{eq:hlp-r3-3}
\begin{align}
\dual{\sigma,\tau}-\dual{u,\ed^{k-1}\tau}&=0 && (\forall \tau\in V^{k-1}),\label{eq:hlp-r3-3a} \\
\dual{\ed^{k-1}\sigma ,v}+\dual{\ed^ku,\ed^kv}+\dual{p,v}&=\dual{f,v}&& (\forall v \in V^k),\label{eq:hlp-r3-3b} \\
\dual{u,q}&=0 && (\forall q\in \mathfrak{H}^k).\label{eq:hlp-r3-3c}
 \end{align}
\end{subequations}

In view of \cite[Theorems 4.7, 4.8 and 4.9]{Arnold(2018)}, we know

\begin{proposition}
\label{pr:h-hl-r3}
These three formulations \eqref{eq:hlp-r3-1}, \eqref{eq:hlp-r3-2}, and \eqref{eq:hlp-r3-3} are all equivalent (see \cite[]{Arnold(2018)}).
There exists a unique solution of the Hodge Laplace problem and that the solution satisfies
\begin{equation}
\label{eq:hlp-r3-5}
\|u\|+\|\ed^ku\|+\|\ed^*_{k}u\|+\|\ed^{k-1}\ed_k^*u\| +
\|\ed_{k+1}^* \ed^ku\|+\|p\|\le c\|f\|
\end{equation}
with a positive constant depending only on the constant $c_p$ appearing in Poincare's inequality.
\end{proposition}

Now, we turn to our steady-state problem \eqref{mainVA-s}. First, the problem \eqref{vecLap-s} and \eqref{mainVAb-s} for finding $\phi$ is nothing but the the Hodge Laplace problem for $k=0$.
For convenience, we assume that $\Omega$ is simply-connected.
Since
$\mathfrak{H}^0=\{0\}$,
\eqref{eq:hlp-r3-1} implies that
\[
 \Delta^0\phi=\ed_1^*\ed^0\phi=-\Div \Grad \phi=\bar{\rho},\quad \phi\in V^0=\mathring{H}^1(\Omega),\quad \ed^0\phi=\Grad u\in V_1^*=H(\Div).
\]

To interpret \eqref{mainVAc-s}--\eqref{mainVAe-s} for finding $A$ in the framework of the Hodge Laplacian problem, we introduce the following problem.

\smallskip

\noindent \textbf{$\mathfrak{B}_k^*$ problem:} Given $g\in \mathfrak{B}_k^*$, find $u\in \mathfrak{B}_k^*$ such that
 \begin{equation}
\label{eq:hlp-r3-11}
\ed_{k+1}^* \ed^{k}u=g.
\end{equation}

In essentially the same way as the proof of \cite[Theorem 4.12]{Arnold(2018)}, we prove

\begin{proposition}
\label{pr:h-bk-r3}
Let $u\in D(\Delta^k)$ be the unique solution of \eqref{eq:hlp-r3-2} for $g\in \mathfrak{B}_k^*$. Then, the function $u$ is in $\mathfrak{B}_k^*$ and it is a solution of \eqref{eq:hlp-r3-11}.
\end{proposition}

The problem \eqref{mainVAc-s}--\eqref{mainVAe-s} for finding $A$ is equivalent to the $\mathfrak{B}_1^*$ problem as long as $\bar{j}$ is taken from $\mathfrak{B}_1^*$. That is, if $\bar{j}$ is given as $\bar{j}=\Rot \tilde{j}$ for some $\tilde{j}\in H(\Rot)$, we have
\[
 \ed_{2}^* \ed^{1}A=\Rot\Rot A=g,\quad A\in V_1=\mathring{H}(\Rot),\quad
\ed_1A=\Rot u\in V_2^*= H(\Rot),%\quad
\]
and
\[
\ed_1^*A=-\Div A=0.
\]
The last assertion follows from $\mathfrak{B}_1^*=\mathcal{R}(d_2^*)\subset\mathcal{N}(\ed_1^*)$.

In summary, the steady-state problem \eqref{mainVA-s} is formulated as
\begin{subequations}
    \label{mainVA-ss}
  \begin{gather}
\phi\in V^0, \quad \ed^0\phi\in V_1^*, \quad
    \ed_1^*\ed^0\phi = \bar{\rho} ,\label{mainVAa-ss}\\
A\in V_1,\quad \ed_1A\in V_2^*,\quad \ed_{2}^* \ed^{1} A = \bar{j},\quad \ed_1^*A = 0. \label{mainVAc-ss}
  \end{gather}
\end{subequations}

Propositions \ref{pr:h-hl-r3} and \ref{pr:h-bk-r3} guarantee that
there exists a unique solution $(\phi,A)\in \mathring{H}^1(\Omega)\times \mathring{H}(\Rot)$ of \eqref{mainVA-ss}, if $\bar{\rho}\in L^2(\Omega)$ and $\bar{j}=\Rot \tilde{j}$ for some $\tilde{j}\in H(\Rot)$.

We proceed to the finite element approximation of the Hodge Laplace problem in the mixed weak form \eqref{eq:hlp-r3-3}. We are interested in the case $k=0$ and $k=1$.
Let $V_h^k$ be a finite dimensional subspace of $V^k$. Then, we have
\[
 \mathfrak{Z}_h^k=\{v\in V_h^k\mid \ed^kv=0\}\subset \mathfrak{Z}^k,\quad
 \mathfrak{B}_h^{k+1}=\{\ed^kv\mid v\in V_h^k\}\subset \mathfrak{B}^{k+1}.
\]
On the other hand, for the discrete harmonic forms
\[
 \mathfrak{H}_h^{k}=\{v\in \mathfrak{Z}_h^{k}
\mid v\bot \mathfrak{B}_h^{k}\},
\]
we do not know whether
\[
 \mathfrak{H}_h^{k}\subset \mathfrak{H}^{k}
\]
holds true or not. However, we know in our setting
\begin{equation}
\label{eq:hh10}
 \mathfrak{H}^{0}=\mathfrak{H}_h^{0}=
 \mathfrak{H}^{1}=\mathfrak{H}_h^{1}=\{0\}
\end{equation}

The Galerkin approximation for \eqref{eq:hlp-r3-3} reads as follows: Given $f\in W^k$, find $\sigma_h\in V_h^{k-1}$,
$u_h\in V_h^k$ and $p_h\in \mathfrak{H}_h^k$ such that
\begin{subequations}
\label{eq:hlp-gakerin-1}
\begin{align}
\dual{\sigma_h,\tau}-\dual{u_h,\ed^{k-1}\tau}&=0 && (\forall \tau\in V_h^{k-1}),\label{eq:hlp-gakerin-1a} \\
\dual{\ed^{k-1}\sigma_h ,v}+\dual{\ed^ku_h,\ed^kv}+\dual{p_h,v}&=\dual{f,v}&& (\forall v \in V_h^k),\label{eq:hlp-gakerin-1b} \\
\dual{u_h,q}&=0 && (\forall q\in \mathfrak{H}_h^k).\label{eq:hlp-gakerin-1c}
 \end{align}
\end{subequations}

In particular, if $k=0$, \eqref{eq:hlp-gakerin-1} implies
: Given $\bar{\rho}\in W^0$, find $u_h\in V_h^k$ such that
%\begin{subequations}
%\label{eq:hlp-gakerin-1}
%\begin{align}
%\dual{\sigma_h,\tau}-\dual{u_h,d^{k-1}\tau}&=0 && (\forall \tau\in V_h^{k-1}),\label{eq:hlp-gakerin-1a} \\
\begin{equation}
\dual{\ed^0u_h,\ed^0v}=\dual{\bar{\rho},v}\qquad  (\forall v \in V_h^0).
\label{eq:hlp-gakerin-1b}%\\
\end{equation}
%\dual{u_h,q}&=0 && (\forall q\in \mathfrak{H}_h^k).\label{eq:hlp-gakerin-1c}  \end{align}
%\end{subequations}
If $k=1$, \eqref{eq:hlp-gakerin-1} implies
: Given $\bar{j}\in \mathfrak{B}_1^*$, find $\sigma_h\in V_h^{0}$,
$u_h\in V_h^1$  such that
\begin{subequations}
\label{eq:hlp-gakerin-2}
\begin{align}
\dual{\sigma_h,\tau}-\dual{u_h,\ed^{0}\tau}&=0 && (\forall \tau\in V_h^{0}),\label{eq:hlp-gakerin-2a} \\
\dual{\ed^{0}\sigma_h ,v}+\dual{\ed^1u_h,\ed^1v}&=\dual{\bar{j},v}&& (\forall v \in V_h^1).\label{eq:hlp-gakerin-2b}
 \end{align}
\end{subequations}
We make the following conditions on $V_h^k$:
\begin{description}
 \item[(H1) Subcomplex property.]We have
$\ed^{k-1}V_h^{k-1} \subset V_h^{k}$ and
$\ed^kV_h^k \subset V_h^{k+1}$. In other words,
 \begin{equation}
\label{eq:sca}
\xymatrix{
  V_h^{k-1} \ar[r]^-{\ed^{k-1}}
  & V_h^k \ar[r]^-{\ed^k}
  & V_h^{k+1}
}
\end{equation}
is a subcomplex of \eqref{eq:dc-r3-1}
 \item[(H2) Existence of bounded cohain projections.]
There exists a linear operator $\pi_h^k:V^k \rightarrow V_h^{k}$ such that
$\ed^k\pi_h^k=\pi_h^{k+1}d^k$ and $\|\pi_h^kv\|_{V^k}\le c\|v\|_{V^k}$, and the restriction of $\pi_h^k$ to $V_h^k$ is the identity on $V_h^k$. In other words, we have the following commuting diagram relating the complex $(V^k,\ed^k)$ to the subcomplex $(V_h^k,\ed^k)$:
 \[
\xymatrix{
   V^k \ar[r]^{\ed^k} \ar[d]^{\pi_h^{k}} &
   V^{k+1} \ar[d]^{\pi_h^{k+1}}
\\
   V_h^k \ar[r]^{\ed^k}&
   V_h^{k+1}
}.
\]
$\pi_h^k$ is assumed that is satisfied bounded with uniformly in h and the commutativity $\pi_h^{k+1} \ed^k = \ed^{k}\pi_h^k$.
 \item[(H3) Approximation property.]
 \begin{equation}
\label{eq:ap1}
\lim_{h\to 0}\inf_{v\in V_h^k}\|w-v\|_{V^k}=0\qquad (w\in V^k).
\end{equation}
\end{description}
Under these assumption, we prove (see \cite[Theorems 5.4 and 5.5]{Arnold(2018)}).
\begin{proposition}
\label{pr:gal-r3}
 Assume that
\textup{(H1)},
\textup{(H2)}, and
\textup{(H3)} are all satisfied.
%  if approximation of complex is satisfied subcomplex and BCP assumption, then following theorem hold
  %\begin{itemize}
   % \item $\mathfrak{H}^k \sim \mathfrak{H}_h^k$, and ${\rm gap}(\mathfrak{H}, \mathfrak{H}_h) \rightarrow 0.$
 %   \item The discrete Poincar{\' e} inequality holds with independent of h.
  %  \item The Galerkin method is consistent and stable.
   % \item The Galerkin method is convergent with the following error estimate.
  %\end{itemize}
Then, \eqref{eq:hlp-gakerin-2} is stable in the sense that
\[
 \|\sigma_h\|+\|u_h\|\le c\|\bar{j}\|
\]
holds true with a positive constant $c$ which is independent of $h$. Moreover, we have
  \begin{equation}
\label{eq:err111}
\lim_{h\to 0}\left(\| \sigma - \sigma_h \|_{V^{0}} +\| u-u_h\|_{V^1}\right)    =0.
  \end{equation}
For \eqref{eq:hlp-gakerin-1b}, we obtain the same conclusions.
\end{proposition}
%\begin{proof}
%Proofs can be found in the original paper
%\cite{Arnold(2006)}\cite{Arnold(2010)}and in the textbook
%
%\end{proof}
Although we do not recall here, many concrete examples of $V_h^k$ satisfying (H1), (H2) and (H3) are known.
\section{Space-time 4D formulation}
\label{SPFormulation}
In the previous section, we reviewed the FEEC frame work using the $L^2$ de Rham complex in $\mathbb{R}^3$.
Based on these preliminaries, we introduce a Hilbert complex in the space-time region in $\mathbb{R}^4$ to handle scalar-valued and vector-valued potentials simultaneously. We then study the proposed Hilbert complex and the Hodge Laplacian and verify that they give a useful framework to solve the periodic steady-state problem \eqref{mainVA-s}.
\subsection{Hilbert complex on $Q$}
In our strategy, the main problem \eqref{mainVA-s} is formulated as a boundary value problem of the Hodge Laplacian problem in the space-time region $Q$. To this end, we write
\[
 \mathbb{R}^4=\mathbb{R}\times \mathbb{R}^3=\{(x_0,x_1,x_2,x_3)
\mid x_0=t\in\mathbb{R},~(x_1,x_2,x_3)\in\mathbb{R}^3\}
\]
and treat $Q=(0,T)\times \Omega$ as a subset of $\mathbb{R}^4$.
We use the dual basis $dx_0,dx_1,dx_2,dx_3$ of the canonical basis $e_0,e_1,e_2,e_3$ and often write $dt=dx_0$.
We introduce a potential $u$ in $Q$ that is a direct sum of the scalar-valued potential $\phi \in W^0=L^2\Lambda_0$ and the vector potential vector-valued potential $A \in W^1=L^2\Lambda_1$ and express it as
\[
 u = dx_0 \wedge \phi+ A,
\]
where the $\wedge$ denotes the wedge product.
Consequently, the potential $u$ is understood as a differential $1$-form on $Q$. We recall that $\phi$ and $A$ are functions of $t$ and $x$. Therefore,
$\phi\in W^0$ should be precisely understood as $\phi(x_0,\cdot)\in W^0$ for any $x_0\in (0,T)$. Similarly,
$A\in W^1$ should be precisely understood as $A(x_0,\cdot)\in W^1$ for any $x_0\in (0,T)$. Below we will employ the abbreviation $\phi \in W^0$ and $A \in W^1$ to express these relations.
Further, the force field $F$ in $Q$ is defined as a direct sum of the electric field $E \in W^1=L^2\Lambda_1$ and the magnetic field $B \in W^2=L^2\Lambda_2$ as $F = dx_0 \wedge E + B$, which is a differential $2$-form on $Q$.

To treat a differential $k$-form in $Q$ of the form $dx_0 \wedge \omega + \omega'$ in a coherent way, we introduce a subset $M_k$ of
a vector space of all differential $k$-forms on $Q$ in the following way. For the time being, we take no care about the smoothness and integrability of differential forms. We set
\begin{subequations}
\label{eq:M}
\begin{align}
 M^0&=W^0 ,\label{eq:M0}\\
 M^{k}&=\{u=dx_0\wedge \omega_{k-1}+\omega_{k}\mid \omega_{k-1}\in W^{k-1},\omega_{k}\in W^{k}\} \qquad (k=1,2,3),\label{eq:Mk}\\
 M^{4}&=\{u=dx_0\wedge \omega_3\mid \omega_3\in W^3\}. \label{eq:M4}
\end{align}
Herein, we recall that the abbreviation $\omega_k=\omega_k(x_0,\cdot)\in W^0$ for any $x_0\in (0,T)$ is employed. Their inner products are defined as
%\begin{align}
%\dual{u,v}_{M^0}&=\int_Q \dual{u,v} \Vol=\int_Qu(x)v(x)~dx_0dx_1dx_2dx_3 ,\label{eq:ipM0}\\
\begin{equation}
  \dual{u,v}_{k}=\int_Q \dual{\omega_{k-1},\eta_{k-1}}\Vol
+\int_Q \dual{\omega_{k},\eta_{k}}\Vol
\qquad (k=0,\ldots,4),\label{eq:ipMk}
\end{equation}
%\end{align}
where
$u=dx_0\wedge \omega_{k-1}+\omega_{k},v=dx_0\wedge \eta_{k-1}+\eta_{k}\in M^{k}$ (with $\omega_{-1}=\eta_{-1}=\omega_{4}=\eta_{4}=0$) and $\Vol=dx_0 \wedge dx_1\wedge dx_2\wedge dx_3$ stands for the volume form on $\mathbb{R}^4$. The norm is defined as
\begin{equation}
  \|u\|_{k}^2=\int_Q \dual{\omega_{k-1},\omega_{k-1}}\Vol
+\int_Q \dual{\omega_{k},\omega_{k}}\Vol
\qquad (k=0,\ldots,4).\label{eq:normMk}
\end{equation}
\end{subequations}
In generally speaking, Minkowsky inner product in the space time leads to indefinite norm and the space is not always Hilbert space \cite{relativity}.
However, the inner product we introduced in 4d space time by enetention of 3d defferntial inner product holds positive defined property.
%\begin{equation}
% \label{eq:r4-form-1}
%M^k=
%\end{equation}
%The scalar-valued potential $\phi$ is the wedge product of $dx_0$ to match the grade.
%Therefore, we consider the following complex by shifting the two components and adding them together.
Then, we can introduce
\begin{equation}
L^2M^k=\{u\in M^k \mid \|u\|_{k}<\infty\}.
\end{equation}
Furthermore, we set for $k=0,\ldots,4$
\begin{subequations}
\label{eq:HM}
\begin{align}
% HM^0&=V^0 ,\label{eq:HM0}\\
%H^*M^0&=V_0^* ,\label{eq:HM0d}\\
 HM^{k}&=\{u=dx_0\wedge \omega_{k-1}+\omega_{k}\mid \omega_{k-1}\in V^{k-1},\omega_{k}\in V^{k}\},\label{eq:HMk} \\
 H^*M^{k}&=\{u=dx_0\wedge \omega_{k-1}+\omega_{k}\mid \omega_{k-1}\in V_{k-1}^*,\omega_{k}\in V_{k}^*\},\label{eq:HMkd}
\end{align}
\end{subequations}
where $\omega_{-1}=\eta_{-1}=\omega_{4}=\eta_{4}=0$.
We now state the definition of linear operators
$\eD^k$ of $L^2M^k\to L^2M^{k+1}$ with its domain $HM^k$ and $\eD_k^*$ of $L^2M^{k}\to L^2M^{k-1}$ with its domain $H^*M^k$:
\begin{subequations}
\label{eq:defDD}
\begin{align}
\eD^k u&=%\eD^k()=
dx_0 \wedge (\ed^{k-1}\omega_{k-1})+\ed^{k}\omega_{k},\label{eq:defDa}\\
\eD_k^* u&=%\eD^k()=
dx_0 \wedge (\ed_{k-1}^*\omega_{k-1})+\ed_{k}^*\omega_{k},\label{eq:defDb}
\end{align}
\end{subequations}
where
$u = dx_0 \wedge \omega_{k-1} + \omega_{k}\in HM^k$ or $u\in H^*M^k$.
Using these operators, the spaces $HM^{k}$ and $H^*M^{k}$ for $k\ge 1$ are expressed alternately as
\begin{subequations}
\label{eq:HM2}
\begin{align}
 HM^{k}&=\{u\in L^2M^{k} \mid \eD^k u\in L^2M^{k+1}\},\label{eq:HM2k} \\
 H^*M^{k}&=\{u\in L^2M^k \mid \eD_{k}^*u\in L^2M^{k-1}\}.\label{eq:HM2kd}
\end{align}
\end{subequations}
These spaces are Hilbert spaces equipped with the following inner products and norms:
\begin{subequations}
\label{eq:HM-norms}
\begin{align}
\dual{u,v}_{HM^k}&=\dual{u,v}_{k}+\dual{\eD^ku,\eD^kv}_{k+1},
\label{eq:HM-norms-a}\\
\|u\|_{HM^k}^2 &=\|u\|_{k}^2+\|\eD^ku\|_{k+1}^2,\label{eq:HM-norms-c}\\
\dual{u,v}_{H^*M^k}&=\dual{u,v}_{k}+\dual{\eD_k^*u,\eD_k^*v}_{k-1},\label{eq:HM-norms-b}\\
\|u\|_{H^*M^k}^2 &=\|u\|_{k}^2+\|\eD_k^*u\|_{k+1}^2.\label{eq:HM-norms-d}
\end{align}
\end{subequations}
The inner products and norms for $HM^0$ and $H^*M^0$ are defined with obvious modifications.
Moreover, $\eD^k$ and $\eD_k^*$ are densely defined closed operators. These properties follows directly from the corresponding properties of $\ed^k$ and $\ed_k^*$.
Then, as a direst consequence of \eqref{eq:dc-r3-1} and \eqref{eq:dc-r3-2}, we have a Hilbert complex,
\begin{equation}
 \label{eq:comp-M-1}
\xymatrix{
  0 \ar[r]
& HM^0\ar[r]^-{\eD^0}
& HM^1\ar[r]^-{\eD^1}
& HM^2\ar[r]^-{\eD^2}
& HM^3\ar[r]^-{\eD^3}
& HM^4 \ar[r] & 0,
}
\end{equation}
and the dual complex,
\begin{equation}
 \label{eq:comp-M-2}
\xymatrix{
 0 & \ar[l] H^*M^0
  & \ar[l]_-{\eD_1^*} H^*M^1
  &\ar[l]_-{\eD_2^*} H^*M^2
  & \ar[l]_-{\eD_3^*} H^*M^3
  & \ar[l]_-{\eD_4^*} H^*M^4
  & \ar[l] 0.
}
\end{equation}
In both complexes, base Hilbert spaces are $L^2M^k$. In particular, we have
$\eD^{k+1}\eD^k=0$ and $\eD_{k}^*\eD_{k+1}^*=0$ as readily obtainable consequences of $\ed^{k+1}\ed^k=0$ and $\ed_{k}^*\ed_{k+1}^*=0$.
In view of Proposition \ref{pr:compact-r4},
$V^{k-1}\cap V_{k-1}^*$ and
$V^{k}\cap V_{k}^*$ are compactly included in $W^{k-1}$ and $W^{k}$, respectively. Therefore, we obtain the following.
\begin{proposition}
\label{pr:compact-r4}
$HM^k\cap H^*M^k$ is compactly included in $L^2M^k$ for $k=0,\ldots,4$.
\end{proposition}
Set
\begin{equation}
 \label{eq:BZ-func1}
\mathcal{B}^k=\mathcal{R}(\eD^{k-1}),\quad
\mathcal{Z}^k=\mathcal{N}(\eD^{k}),\quad
\mathcal{B}_k^*=\mathcal{R}(\eD_{k+1}^*),\quad
\mathcal{Z}_k^*=\mathcal{N}(\eD_{k}^*).
\end{equation}
An element $u\in L^2M^k$ is called a harmonic $k$-form, if $\eD^ku=0$ and $\eD_{k}^*u=0$, and the set of all harmonic $k$-forms is denoted by $\mathcal{H}^k$. We have
\[
 \mathcal{H}^k
=\{u=dx_0\wedge \omega_{k-1}+\omega_k \mid \omega_{k-1}\in \mathfrak{H}^{k-1},~
\omega_{k}\in \mathfrak{H}^{k}\}.
\]
Theretofore, in our setting (see \eqref{eq:hh}),
\begin{equation}
\label{eq:hf-r4}
\mathcal{H}^1=\{0\}.
\end{equation}
Moreover, as a consequence of Proposition \ref{pr:h-poin-r3}, we have
the Poincar{\' e} inequality as
  \begin{equation}
\label{eq:po-r4}
    \| u \|_{H{M}^k} \leq C_P \| \eD^ku\|_k\qquad (u \in (\mathfrak{Z}^k )^{\perp} \cap HM^k).
  \end{equation}

\subsection{The periodic steady-state problem \eqref{mainVA}}
\label{subse:main}
At this stage, we introduce the Hodge Laplacian $\eL^k:L^2M^k\to L^2M^k$ as
\begin{subequations}
\label{eq:hl-r4}
\begin{equation}
\label{eq:hl-r4-a}
\eL^k=\eD^{k-1}\eD_k^*+\eD_{k+1}^*\eD^{k}
\end{equation}
with its domain
\begin{equation}
\label{eq:hl-r4-b}
D(\eL^k)=\{u\in HM^k\cap H^*M^k\mid \eD^2u\in H^*M^{k+1},~\eD_k^* u\in HM^{k-1}\}.
\end{equation}
\end{subequations}
The Hodge Laplace problem in a strong form is described as follows:
Given $F\in L^2M^k$, find $u\in D(\eL^k)$ such that
 \begin{equation}
\label{eq:hlp-r4-1}
\eL^ku=F-P_\mathcal{H}F,\quad
u\bot \mathcal{H}^k,
% = d^{k-1}d_k^* + d_{k+1}^* d^k
\end{equation}
where $P_\mathcal{H}$ denotes the orthogonal projection form $L^2M^k$ onto $\mathcal{H}^k$.
%A primal weak and mixed weak are also defined form is:
%Given $f\in W^k$, find $u\in V^k\cap V_k^*$ such that $u\bot \mathfrak{H}^k$ and
%\begin{equation}
%\label{eq:hlp-r3-2}
%\dual{d^ku,d^kv}+\dual{d^*_ku,d^*_kv}=\dual{f-P_{\mathfrak{H}}f,v}\quad (\forall v \in V^k\cap V_k^*).
%\end{equation}
We skip a primal weak form and state
the Hodge Laplace problem in a mixed weak form:
Given $F\in L^2M^k$, find $\sigma\in HM^{k-1}$,
$u\in HM^k$ and $p\in \mathcal{H}^k$ such that %$u\bot \mathfrak{H}^k$ and
\begin{subequations}
\label{eq:hlp-r4-3}
\begin{align}
\dual{\sigma,\tau}_{k-1}-\dual{u,\eD^{k-1}\tau}_k&=0 && (\forall \tau\in HM^{k-1}),\label{eq:hlp-r4-3a} \\
\dual{\eD^{k-1}\sigma ,v}_k+\dual{\eD^ku,\eD^kv}_{k+1}+\dual{p,v}_{k}&=\dual{F,v}_k&& (\forall v \in HM^k),\label{eq:hlp-r4-3b} \\
\dual{u,q}_k&=0 && (\forall q\in \mathcal{H}^k).\label{eq:hlp-r4-3c}
 \end{align}
\end{subequations}
The following proposition is an application of
\cite[Theorems 4.7, 4.8 and 4.9]{Arnold(2018)} as Proposition \ref{pr:h-hl-r3}.
\begin{proposition}
\label{pr:h-hl-r4}
Two formulations \eqref{eq:hlp-r4-1} and \eqref{eq:hlp-r4-3} are equivalent.
There exists a unique solution of the Hodge Laplace problem and that the solution satisfies
\begin{equation}
\label{eq:hlp-r4-5}
\|u\|_k+\|\eD^ku\|_{k+1}+\|\eD^*_{k}u\|_{k-1}+\|\eD^{k-1}\eD_k^*u\|_k +
\|\eD_{k+1}^* \eD^ku\|_k+\|p\|_k\le c\|F\|_k
\end{equation}
with a positive constant $c$ depending only on the constant appearing in Poincare's inequality.
\end{proposition}
As the steady-state problem \eqref{mainVA-s} is equivalent to the $\mathfrak{B}_1^*$-problem of \eqref{eq:hlp-r3-1}, the periodic steady-state problem \eqref{mainVA} is equivalent to the $\mathcal{B}_1^*$ problem of the Hodge Laplace problem \eqref{eq:hlp-r4-1}. We explain this fact more precisely.
Letting $\bar{\rho}\in W^0$ and $\bar{j}\in W^1$, we set $F=dx_0\wedge \bar{\rho}+\bar{j}\in L^2M^1$. Assume that $F\in\mathcal{B}_1^*=\mathcal{R}(\eD_2^*)$. That is, we assume that $\bar{\rho}$ and $\bar{j}$ are expressed as $\bar{\rho}=\ed_1^*\tilde{\rho}$ and
$\bar{j}=\ed_2^*\tilde{j}$ for some
$\tilde{\rho}\in V_1^*$ and
$\tilde{j}\in V_2^*$. By Proposition \ref{pr:h-hl-r4}, there exists a unique $v\in D(\eL^2)$ satisfying $\eL^2 v=\tilde{F}:=dx_0\wedge \tilde{\rho}+\tilde{j}\in HM^2$.
Setting
\begin{subequations}
\begin{equation}
 \label{eq:b1-z}
 u=\eD_2^*v,% \in H^*M^1,
\end{equation}
we have
\begin{align}
 \eL^1 u
&=(\eD^0\eD_1^*+\eD_2^*\eD^1)\eD_2^* v \label{eq:b1-a}\\
&=\eD_2^*\eD^1\eD_2^* v  \label{eq:b1-b}\\
&=\eD_2^*(\eD^1\eD_2^*+\eD_3^*\eD^2)v
=\eD_2^*\eL^2v=\eD_2^*\tilde{F}=F .\label{eq:b1-c}
\end{align}
\end{subequations}
By \eqref{eq:b1-z}, \eqref{eq:b1-b} and \eqref{eq:b1-c}, we find that $u\in\mathcal{B}_1^*=\mathcal{R}(\eD_2^*)$ and it solves
\begin{equation}
 \label{eq:b1-w}
 \eD_2^*\eD^1u=F.
\end{equation}
This implies that $u=dx_0\wedge \phi+A$ is a solution of
\begin{subequations}
    \label{mainVA-ss2}
  \begin{gather}
\phi\in V^0, \quad \ed^0\phi\in V_1^*, \quad
    \ed_1^*\ed^0\phi = \bar{\rho} ,\label{mainVAa-ss}\\
A\in V^1,\quad \ed_1A\in V_2^*,\quad \ed_{2}^* \ed^{1} A = \bar{j},\quad \ed_1^*A = 0 \label{mainVAc-ss}
  \end{gather}
\end{subequations}
for any $x_0\in (0,T)$. That is, $(\phi,A)$ is a solution of \eqref{mainVA}.
\section{Finite element approximation}
\label{sec:fem}
In the previous section, we formulate the periodic steady-state problem \eqref{mainVA} as the the Hodge Laplace problem \eqref{eq:b1-w} for the differential $1$-form $u$ in $Q\subset\mathbb{R}^4$. Then, we can apply the abstract theory recalled in \S \ref{sec:FEEC-R3} for the Galerkin approximation. The only thing we leave is to construct concretely a finite dimensional subspace $V_h^k$ of $HM^k$ which satisfy (H1), (H2) and (H3) in \S \ref{sec:FEEC-R3}.
In this section, we assume that $\Omega$ is a $3$-rectangle in $\mathbb{R}^3$ and that we are given a mesh subdivision $\mathcal{T}_h$ of $(0,T)\times\Omega$ composed of $4$-rectangle elements. The size parameter $h$ is defined as the maximum length of each $K\in\mathcal{T}_h$.
\subsection{Approximation by a cubical element}
Let $n\ge 1$ and $r\ge 1$ be integers. We introduce the cubical element (see \cite{abf15})
\begin{equation}
    Q_r^-\Lambda^k(I^n) = \bigoplus_{1\leq \sigma_1 <\sigma_2< \cdots <\sigma_k\le n}\left[\bigotimes_{i=1}^n \mathcal{P}_{r-\delta_{i,\sigma}}(I)\right]
    dx_{\sigma_1}\wedge \cdots \wedge dx_{\sigma_k},
\end{equation}
  where $\mathcal{P}_r(I)$ denotes a set of all polynomial defined in $I=[0,1]$ of degree $\le r$ and
  \begin{equation}
    \delta_{i,\sigma}=
    \begin{cases}
   1& (i\in\{\sigma_1,\cdots \sigma_k\}) ,\\
    0& (\mbox{otherwise}).
    \end{cases}
  \end{equation}
Using this, we set
\begin{multline}
H_h{M}^k(I^n)=\{dx_0\wedge (\bar{\eta}(x_0)\omega^{k-1})+\omega^{k}\mid
\\
\omega^{k-1}\in Q_r^-\Lambda^{k-1}(I^3), \bar{\eta}(x^0) \in P_{r-1}(I),
 \omega^{k}\in Q_r^-\Lambda^{k}(I^3), {\eta}(x^0) \in P_{r}(I)
\}
\label{eq:HhMI}
\end{multline}
with $\omega^{-1}=\omega^4=0$.
Then, $Q_r^-\Lambda^k(K)$ and $H_h{M}^k(K)$ are defined similarly for $K\in \mathcal{T}_h$. Actually, they correspond to the case $n=4$.
\begin{theorem}
\label{th:SCP-r4-1}
The space $H_h{M}^k(K)$ can be identified with $Q_r^-\Lambda^k(K)$ for any $K\in \mathcal{T}_h$.
% The approximate complex introduced in this study is equivalent to the conventional cubical element complex in the hypercube $I^4 \subset \mathbb{R}^4$.
%  \begin{equation}
%      H_h{M}(I^4) = Q_r^-\Lambda^k(I^4)
%  \end{equation}
%Furthermore, $H_h{M}^k \subset H{M}^k$, $H_h{M}^{k+1} = D^{k}H_h{M}^{k} $
\end{theorem}
\begin{proof}
It is verified by a direct calculation. For example,
%It suffies to prove the case $K=I^4$.
%By comparing each component, the following can be confirmed.
%(k=0)\\
%\begin{align}
%  u_h^0 = \eta(x_0) \alpha(x_1) \beta(x_2) \gamma(x_3)
%\end{align}
%(k=1)\\
%\begin{align}
%  u_h^1 &= \bar{\eta}(x_0) \alpha(x_1) \beta(x_2) \gamma(x_3)dx_0 \notag\\
%        &+ \eta(x_0) \bar{\alpha}(x_1) \beta(x_2) \gamma(x_3)dx_1 \notag\\
%        &+ \eta(x_0) \alpha(x_1) \bar{\beta}(x_2) \gamma(x_3)dx_2 \notag\\
%        &+ \eta(x_0) \alpha(x_1) \beta(x_2) \bar{\gamma}(x_3)dx_3
%\end{align}
%(k=2)\\
%\begin{align}
$u_h\in H_hM^2(K)$ is expressed as
\begin{multline*}
  u_h = \bar{\eta}(x_0) \bar{\alpha}(x_1) \beta(x_2) \gamma(x_3)
        dx_0\wedge dx_1
        + \bar{\eta}(x_0) \alpha(x_1) \bar{\beta}(x_2) \gamma(x_3)
        dx_0\wedge dx_2 \\
        + \bar{\eta}(x_0) \alpha(x_1) \bar{\beta}(x_2) \bar{\gamma}(x_3)
        dx_0\wedge dx_3
        + \eta(x_0) \alpha(x_1) \bar{\beta}(x_2) \bar{\gamma}(x_3)
        dx_2\wedge dx_3 \\
        + \eta(x_0) \bar{\alpha}(x_1) \beta(x_2) \bar{\gamma}(x_3)
        dx_3\wedge dx_1
        + \eta(x_0) \bar{\alpha}(x_1) \bar{\beta}(x_2) \gamma(x_3)
        dx_1\wedge dx_2,
\end{multline*}
where $\eta(x_0), \alpha(x_1), \beta(x_2), \gamma(x_3) \in \mathcal{P}_r(I)$ and $\bar{\eta}(x_0), \bar{\alpha}(x_1), \bar{\beta}(x_2), \bar{\gamma}(x_3) \in \mathcal{P}_{r-1}(I)$. Therefore, $u_h\in Q_r^-\Lambda^2(K)$. The converse is the same.
%\end{align}
%(k=3)\\
%\begin{align}
%  u_h^3 &= \bar{\eta}(x_0) \bar{\alpha}(x_1) \bar{\beta}(x_2) \gamma(x_3)
%        dx_0\wedge dx_1 \wedge dx_2 \notag\\
%        &+ \bar{\eta}(x_0) \bar{\alpha}(x_1) \beta(x_2) \bar{\gamma}(x_3)
 %       dx_3\wedge dx_0 \wedge dx_1 \notag\\
  %      &+ \bar{\eta}(x_0) \alpha(x_1) \bar{\beta}(x_2) \bar{\gamma}(x_3)
   %     dx_2\wedge dx_3 \wedge dx_0 \notag\\
    %    &+ \eta(x_0) \bar{\alpha}(x_1) \bar{\beta}(x_2) \bar{\gamma}(x_3)
     %   dx_1\wedge dx_2 \wedge dx_3
%\end{align}
%(k=4)
%\begin{align}
%u_h^4 = \bar{\eta}(x_0) \bar{\alpha}(x_1) \bar{\beta}(x_2) \bar{\gamma}(x_3)
%dx_0 \wedge dx_1\wedge dx_2 \wedge dx_3
%\end{align}
\end{proof}
We now introduce a finite element space for $HM^k$ as
\begin{equation}
 \label{eq:HhMk-r4-1}
%V_h^k=
H_h{M}^k=\{u_h\in HM^k \mid u|_{K}\in H_hM^k(K)\ (K\in\mathcal{T}_h)\}.
\end{equation}
We have
\[
 \mathcal{Z}_h^k=\{v\in H_hM^k\mid \eD^kv=0\}\subset \mathcal{Z}^k,\quad
 \mathcal{B}_h^{k+1}=\{\eD^kv\mid v\in H_hM^k\}\subset \mathcal{B}^{k+1}
\]
and the discrete harmonic form is defined as
$\mathcal{H}_h^{k}=\{v\in \mathcal{Z}_h^{k}
\mid v\bot \mathcal{B}_h^{k}\}$.
In our setting,
\begin{equation}
\label{eq:hh10-r4}
 \mathcal{H}^{1}=\mathcal{H}_h^{1}=\{0\}.
\end{equation}
Then, the finite element approximation of \eqref{eq:hlp-r4-3} reads as follows: Given $F\in L^2M^k$, find $\sigma\in H_hM^{k-1}$,
$u_h\in H_hM^k$ and $p_h\in \mathcal{H}_h^k$ such that %$u\bot \mathfrak{H}^k$ and
\begin{subequations}
\label{eq:gal-r4-3}
\begin{align}
\dual{\sigma_h,\tau}_{k-1}-\dual{u_h,\eD^{k-1}\tau}_k&=0 && (\forall \tau\in H_hM^{k-1}),\label{eq:gal-r4-3a} \\
\dual{\eD^{k-1}\sigma_h ,v}_k+\dual{\eD^ku_h,\eD^kv}_{k+1}+\dual{p_h,v}_{k}&=\dual{F,v}_k&& (\forall v \in H_hM^k),\label{eq:gal-r4-3b} \\
\dual{u_h,q}_k&=0 && (\forall q\in \mathcal{H}_h^k).\label{eq:gal-r4-3c}
 \end{align}
\end{subequations}
\begin{theorem}
\label{th:SCP-r4-2}
\begin{enumerate}
\item[(a)]
The space $H_h{M}^k$ has the approximation property;
\begin{equation*}
%\label{eq:ap1-r4}
\lim_{h\to 0}\inf_{v\in H_hM^k}\|w-v\|_{HM^k}=0\qquad (w\in HM^k).
\end{equation*}
\item[(b)] The space $H_h{M}^k$ has the subcomplex property;
$\eD^{k-1}H_hM^{k-1} \subset H_hM^{k}$ and
$\eD^kH_hM^k \subset H_hM^{k+1}$.
\item[(c)]
There exists a bounded cohain projection $\Pi_h^k:HM^k \rightarrow H_hM^{k}$; The following diagram commutes:
\[
\xymatrix{
  0 \ar[r] &
   H{M}^0 \ar[r]^{\eD^0} \ar[d]^{\Pi_h^0} &
   H{M} \ar[r]^{\eD^1} \ar[d]^{\Pi_h^1}&
   H{M} \ar[r]^{\eD^2} \ar[d]^{\Pi_h^2}&
   H{M} \ar[r]^{\eD^3} \ar[d]^{\Pi_h^3}&
   H{M} \ar[r]^{\eD^4} \ar[d]^{\Pi_h^4}&
  0\\
  0 \ar[r] &
   H_h{M}^0 \ar[r]^{\eD^0}&
   H_h{M}^1 \ar[r]^{\eD^1}&
   H_h{M}^2 \ar[r]^{\eD^2}&
   H_h{M}^3 \ar[r]^{\eD^3}&
   H_h{M}^4 \ar[r]^{\eD^4}&
  0.
}
\]
\end{enumerate}
%\end{subequations}
\end{theorem}
\begin{proof}
(a) is a standard fact. (b) follows from a direct calculation.
(c) is a consequence of the result for $n=3$ as verified below.
Let $\mathcal{T}'_h$ be a subdivision of $\Omega$ by $3$-rectangles such that $\mathcal{T}_h|_{x_0=0}=\mathcal{T}_h'$, we set
\[
 % \label{eq:HhMk-r4-1}
V_h^k
%H_h{M}^k
=\{v_h\in V^k \mid v|_{K'}\in H_hM^k(K')\ (K'\in\mathcal{T}_h')\}.
\]
According to the explanation in for $\mathcal{P}_r^-$ \cite[p.92]{Arnold(2018)}.
The paper \cite{abf15} shows the existence of bounded cochain projection following the method of \cite{smoothing}.
The case of $Q_r^-\Lambda^k$ is also shown to have the bounded cochain projection by exactly same procedure.
Therefore, there exists a bounded cochain projection  where
 $\pi_h:V^k \to V_h^k$ and the cubical element satisfies
the commutativity property: $\pi_h^{k+1}\ed^k = \ed^{k}\pi_h^k$. We set
\[
 \Pi_hu=dx_0\wedge \pi_h^{k-1}\omega^{k-1}+\pi_h^{k}\omega^k\quad
(u=dx_0\wedge\omega^{k-1}+\omega^k\in HM^k).
\]
Then, we have $\Pi_h^{k+1}\eD^k = \eD^k\Pi_h^k$ by a direct calculation.
\end{proof}
Therefore, we obtain (see \cite[Theorems 5.4 and 5.5]{Arnold(2018)})
\begin{theorem}
\label{pr:gal-r4}
The finite element scheme \eqref{eq:gal-r4-3} is stable and convergent in the sense of Proposition \ref{pr:gal-r3}.
\end{theorem}
\subsection{Reference element}
As a concrete example, consider a hypercube with a node element and an edge element as the reference elements in the 4d space-time (see Figure \ref{4delement}).
A cube is placed at time $T_0$.
This cube is extruded to $T_1$ along with time direction to make a hypercube.
The number of node points is 8+8=16.
The cube in the time $T_0$ includes 12 edges, and the cube in the time $T_1$ also consists of 12 edges. Besides, extruded eight nodes make eight edges along with time direction, so the total number of the edge is 32.
Hypersurfaces and hypervolumes are also considered in the same way, with 20,1.
We consider the differential forms
,0-form $\sigma_h \in Q_1^-\Lambda^0(\bar{K})=Q_1(\bar{K})$ and a 1-form $u_h \in Q_1^-\Lambda^1(\bar{K})$, on the reference element
$
\bar{K}= \{(t,x,y,z) ;-\Delta T/2  \leq t \leq \Delta T/2,
-\Delta X/2  \leq x \leq \Delta X/2,
-\Delta Y/2  \leq y \leq \Delta Y/2,
-\Delta Z/2  \leq z \leq \Delta Z/2\}
$.
A hyper node reference element have values on the 16 grid points.
$\sigma_h$ is considered as following
\begin{gather}
\sigma_h = \sum_{i=1}^{16}\sigma^i P_i(x,y,z,t)\\
P_{1} = \frac{1}{16}(1-\frac{2x}{\Delta X})(1-\frac{2y}{\Delta Y})(1-\frac{2z}{\Delta Z})(1-\frac{2t}{\Delta T}),\notag\\
P_{2} = \frac{1}{16}(1+\frac{2x}{\Delta X})(1-\frac{2y}{\Delta Y})(1-\frac{2z}{\Delta Z})(1-\frac{2t}{\Delta T}),\notag\\
P_{3} = \frac{1}{16}(1-\frac{2x}{\Delta X})(1+\frac{2y}{\Delta Y})(1-\frac{2z}{\Delta Z})(1-\frac{2t}{\Delta T}),\notag\\
P_{4} = \frac{1}{16}(1+\frac{2x}{\Delta X})(1+\frac{2y}{\Delta Y})(1-\frac{2z}{\Delta Z})(1-\frac{2t}{\Delta T}),\notag\\
P_{5} = \frac{1}{16}(1-\frac{2x}{\Delta X})(1-\frac{2y}{\Delta Y})(1+\frac{2z}{\Delta Z})(1-\frac{2t}{\Delta T}),\notag\\
P_{6} = \frac{1}{16}(1+\frac{2x}{\Delta X})(1-\frac{2y}{\Delta Y})(1+\frac{2z}{\Delta Z})(1-\frac{2t}{\Delta T}),\notag\\
P_{7} = \frac{1}{16}(1-\frac{2x}{\Delta X})(1+\frac{2y}{\Delta Y})(1+\frac{2z}{\Delta Z})(1-\frac{2t}{\Delta T}),\notag\\
P_{8} = \frac{1}{16}(1+\frac{2x}{\Delta X})(1+\frac{2y}{\Delta Y})(1+\frac{2z}{\Delta Z})(1-\frac{2t}{\Delta T}),\notag\\
P_{9} = \frac{1}{16}(1-\frac{2x}{\Delta X})(1-\frac{2y}{\Delta Y})(1-\frac{2z}{\Delta Z})(1+\frac{2t}{\Delta T}),\notag\\
P_{10} = \frac{1}{16}(1+\frac{2x}{\Delta X})(1-\frac{2y}{\Delta Y})(1-\frac{2z}{\Delta Z})(1+\frac{2t}{\Delta T}),\notag\\
P_{11} = \frac{1}{16}(1-\frac{2x}{\Delta X})(1+\frac{2y}{\Delta Y})(1-\frac{2z}{\Delta Z})(1+\frac{2t}{\Delta T}),\notag\\
P_{12} = \frac{1}{16}(1+\frac{2x}{\Delta X})(1+\frac{2y}{\Delta Y})(1-\frac{2z}{\Delta Z})(1+\frac{2t}{\Delta T}),\notag\\
P_{13} = \frac{1}{16}(1-\frac{2x}{\Delta X})(1-\frac{2y}{\Delta Y})(1+\frac{2z}{\Delta Z})(1+\frac{2t}{\Delta T}),\notag\\
P_{14} = \frac{1}{16}(1+\frac{2x}{\Delta X})(1-\frac{2y}{\Delta Y})(1+\frac{2z}{\Delta Z})(1+\frac{2t}{\Delta T}),\notag\\
P_{15} = \frac{1}{16}(1-\frac{2x}{\Delta X})(1+\frac{2y}{\Delta Y})(1+\frac{2z}{\Delta Z})(1+\frac{2t}{\Delta T}),\notag\\
P_{16} = \frac{1}{16}(1+\frac{2x}{\Delta X})(1+\frac{2y}{\Delta Y})(1+\frac{2z}{\Delta Z})(1+\frac{2t}{\Delta T})
\end{gather}
\begin{figure}[htbp]
 \centering
 \includegraphics[width=1.0\textwidth]{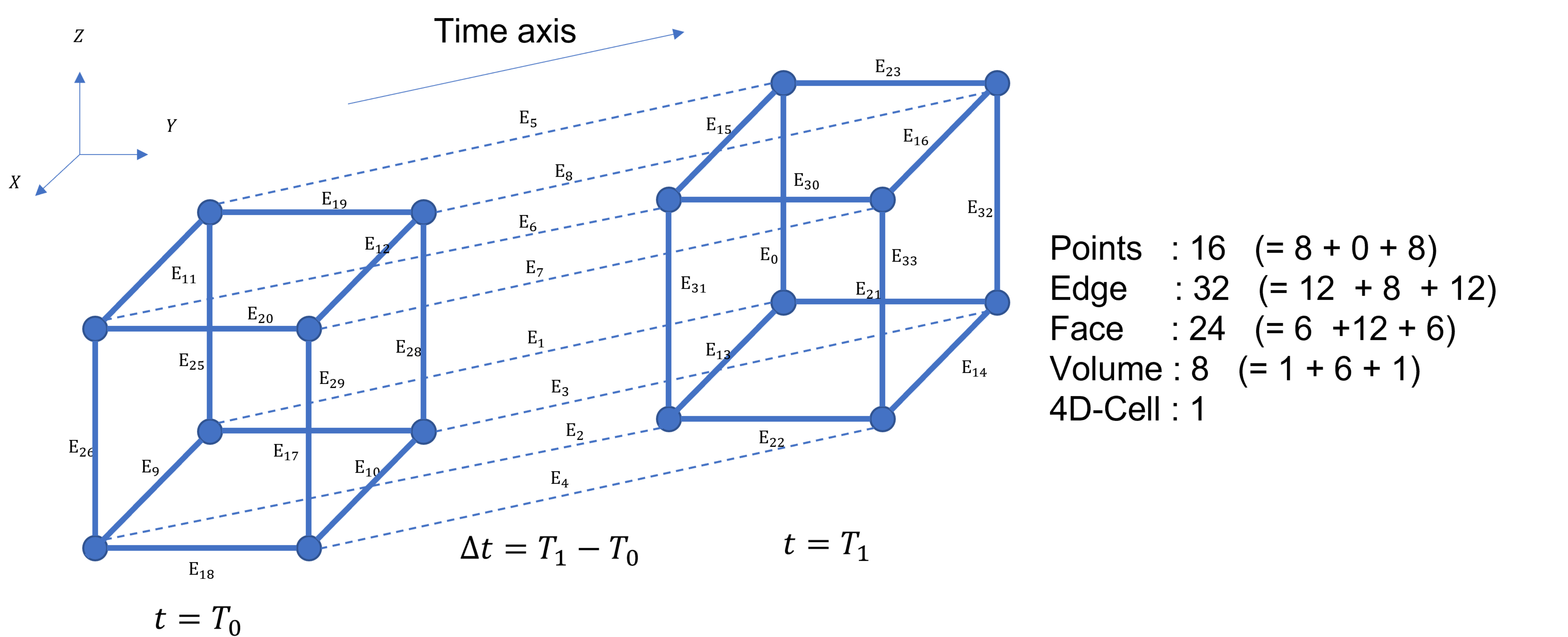}
\caption{hyper edge element}
  \label{4delement}
\end{figure}
A hyper edge reference element have values on the 32 edges.
$u_h$ is considered as following
\begin{equation}
  u_h = \phi dt + A_1 dx + A_2 dy + A_3 dz
\end{equation}
which,
\begin{gather}
  \phi = \sum_{i=1}^8 \phi^i E_i(x,y,z,t)\\
  E_1 = \frac{1}{8}(1-\frac{2x}{\Delta X})(1-\frac{2y}{\Delta Y})(1-\frac{2z}{\Delta Z}),
  E_2 = \frac{1}{8}(1+\frac{2x}{\Delta X})(1-\frac{2y}{\Delta Y})(1-\frac{2z}{\Delta Z}),\notag\\
  E_3 = \frac{1}{8}(1-\frac{2x}{\Delta X})(1+\frac{2y}{\Delta Y})(1-\frac{2z}{\Delta Z}),
  E_4 = \frac{1}{8}(1+\frac{2x}{\Delta X})(1+\frac{2y}{\Delta Y})(1-\frac{2z}{\Delta Z}),\notag\\
  E_5 = \frac{1}{8}(1-\frac{2x}{\Delta X})(1-\frac{2y}{\Delta Y})(1+\frac{2z}{\Delta Z}),
  E_6 = \frac{1}{8}(1+\frac{2x}{\Delta X})(1-\frac{2y}{\Delta Y})(1+\frac{2z}{\Delta Z}),\notag\\
  E_7 = \frac{1}{8}(1-\frac{2x}{\Delta X})(1+\frac{2y}{\Delta Y})(1+\frac{2z}{\Delta Z}),
  E_8 = \frac{1}{8}(1+\frac{2x}{\Delta X})(1+\frac{2y}{\Delta Y})(1+\frac{2z}{\Delta Z})
\end{gather}
\begin{gather}
  A_1 = \sum_{i=9}^{16} A^i E_i(x,y,z,t)  \\
  E_9 = \frac{1}{8}(1-\frac{2y}{\Delta Y})(1-\frac{2z}{\Delta Z})(1-\frac{2t}{\Delta T}),
  E_{10} = \frac{1}{8}(1+\frac{2y}{\Delta Y})(1-\frac{2z}{\Delta Z})(1-\frac{2t}{\Delta T}),\notag\\
  E_{11} = \frac{1}{8}(1-\frac{2y}{\Delta Y})(1+\frac{2z}{\Delta Z})(1-\frac{2t}{\Delta T}),
  E_{12} = \frac{1}{8}(1+\frac{2y}{\Delta Y})(1+\frac{2z}{\Delta Z})(1-\frac{2t}{\Delta T}),\notag\\
  E_{13} = \frac{1}{8}(1-\frac{2y}{\Delta Y})(1-\frac{2z}{\Delta Z})(1+\frac{2t}{\Delta T}),
  E_{14} = \frac{1}{8}(1+\frac{2y}{\Delta Y})(1-\frac{2z}{\Delta Z})(1+\frac{2t}{\Delta T}),\notag\\
  E_{15} = \frac{1}{8}(1-\frac{2y}{\Delta Y})(1+\frac{2z}{\Delta Z})(1+\frac{2t}{\Delta T}),
  E_{16} = \frac{1}{8}(1+\frac{2y}{\Delta Y})(1+\frac{2z}{\Delta Z})(1+\frac{2t}{\Delta T})
  \end{gather}
  \begin{gather}
  A_2 = \sum_{i=17}^{24} A^i E_i(x,y,z,t)\\
  E_{17} = \frac{1}{8}(1-\frac{2z}{\Delta Z})(1-\frac{2x}{\Delta X})(1-\frac{2t}{\Delta T}),
  E_{18} = \frac{1}{8}(1+\frac{2z}{\Delta Z})(1-\frac{2x}{\Delta X})(1-\frac{2t}{\Delta T}),\notag\\
  E_{19} = \frac{1}{8}(1-\frac{2z}{\Delta Z})(1+\frac{2x}{\Delta X})(1-\frac{2t}{\Delta T}),
  E_{20} = \frac{1}{8}(1+\frac{2z}{\Delta Z})(1+\frac{2x}{\Delta X})(1-\frac{2t}{\Delta T}),\notag\\
  E_{21} = \frac{1}{8}(1-\frac{2z}{\Delta Z})(1-\frac{2x}{\Delta X})(1+\frac{2t}{\Delta T}),
  E_{22} = \frac{1}{8}(1+\frac{2z}{\Delta Z})(1-\frac{2x}{\Delta X})(1+\frac{2t}{\Delta T}),\notag\\
  E_{23} = \frac{1}{8}(1-\frac{2z}{\Delta Z})(1+\frac{2x}{\Delta X})(1+\frac{2t}{\Delta T}),
  E_{24} = \frac{1}{8}(1+\frac{2z}{\Delta Z})(1+\frac{2x}{\Delta X})(1+\frac{2t}{\Delta T})
  \end{gather}
  \begin{gather}
  A_3 = \sum_{i=25}^{32} A^i E_i(x,y,z,t)  \\
  E_{25} = \frac{1}{8}(1-\frac{2x}{\Delta X})(1-\frac{2y}{\Delta Y})(1-\frac{2t}{\Delta T}),
  E_{26} = \frac{1}{8}(1+\frac{2x}{\Delta X})(1-\frac{2y}{\Delta Y})(1-\frac{2t}{\Delta T}),\notag\\
  E_{27} = \frac{1}{8}(1-\frac{2x}{\Delta X})(1+\frac{2y}{\Delta Y})(1-\frac{2t}{\Delta T}),
  E_{28} = \frac{1}{8}(1+\frac{2x}{\Delta X})(1+\frac{2y}{\Delta Y})(1-\frac{2t}{\Delta T}),\notag\\
  E_{29} = \frac{1}{8}(1-\frac{2x}{\Delta X})(1-\frac{2y}{\Delta Y})(1+\frac{2t}{\Delta T}),
  E_{30} = \frac{1}{8}(1+\frac{2x}{\Delta X})(1-\frac{2y}{\Delta Y})(1+\frac{2t}{\Delta T}),\notag\\
  E_{31} = \frac{1}{8}(1-\frac{2x}{\Delta X})(1+\frac{2y}{\Delta Y})(1+\frac{2t}{\Delta T}),
  E_{32} = \frac{1}{8}(1+\frac{2x}{\Delta X})(1+\frac{2y}{\Delta Y})(1+\frac{2t}{\Delta T})
\end{gather},respectively.
The variable of scalar potential is placed on an edge in the time direction and the variable does not contain time.In higher-order elements, time is included at (r-1)-order.
The vector potential has a time component, but the differential operator $d^k$  not contribute for time.
\subsection{Arrow--Hurwicz Algorithm (AHA)}
\label{sec:iterative}
As is stated in \S \ref{subse:main}, the periodic steady-state problem \eqref{mainVA} is equivalent to the $\mathcal{B}_1^*$ problem of the Hodge Laplace problem \eqref{eq:hlp-r4-1};
\begin{equation}
 \label{eq:b1-w10}
u=dx_0\wedge \phi+A\in \mathcal{B}_1^*,\quad
 \eD_2^*\eD^1u=F,
\end{equation}
where $F=dx_0\wedge \bar{\rho}+\bar{j}\in\mathcal{B}_1^*=\mathcal{R}(\eD_2^*)$. %That is, we assume that $\bar{\rho}$ and $\bar{j}$ are expressed as $\bar{\rho}=\ed_0^*\tilde{\rho}$ and
%$\bar{j}=\ed_1^*\tilde{j}$ for some
%$\tilde{\rho}\in V_1^*$ and
%$\tilde{j}\in V_2^*$.
We recall that $\mathcal{B}_1^*$ is defined as $ \mathcal{B}_1^*=\mathcal{R}(\eD_2^*)\subset\mathcal{N}(\eD_1^*)$.
If $F\in \mathcal{B}_1^*$ is given exactly (in a suitable discrete sense), we can obtain the solution $u$ by solving the full Hodge Laplace problem \eqref{eq:hlp-r4-1}. However, it is difficult to ensure $F\in \mathcal{B}_1^*$ in the actual computation. Therefore, it is useful to formulate \eqref{eq:b1-w10} as a saddle point problem of the action $\mathcal{S}$ defined as
\begin{equation}
 \label{eq:aha01}
\mathcal{S}(v,\tau) = \frac{1}{2}\dual{\eD^1 v, \eD^1 v}_{2} - \dual{v,\eD^0\tau}_{1} - \dual{F,v}_1\quad
(v\in HM^1,\tau\in HM^0).
\end{equation}
Recall that $(u,\sigma)\in HM^1\times HM^0$ is called a saddle point of $\mathcal{S}$ if it satisfies
\begin{equation}
 \label{eq:aha05}
 \mathcal{S}(u, \tau) \leq \mathcal{S}(u, \sigma) \leq \mathcal{S}(v,\sigma)
\qquad (\forall (v,\tau)\in HM^1\times HM^0).
\end{equation}
%According to FEEC theory, the $\mathfrak{B}^*$ problem requires $F \in \mathfrak{B}^*$.
%Since the actual problem F is not always included in $³{B}^*$, if it is not%, we can solve a mixed weak form subproblem of the full Hodge-Laplacian problem and subtract PF from F in the pre-problem.
%This time we assume $F \in \mathfrak{B}^*$.
%The following theorem confirms that the is the $\mathfrak{B}^*$ problem of the Hodge Laplacian in four-dimensional space-time.
To be more specific, we state the following theorem.
\begin{theorem}
A couple of differential forms $(u,\sigma)\in HM^1\times HM^0$ is a  a saddle point of $\mathcal{S}$ if and only if it solves the mixed weak form of \eqref{eq:b1-w10} given as
\begin{subequations} % 2021-11-10 15:37の式群
\label{eq:aha02}
 \begin{align}
      \dual{\eD^1 u, \eD^1 v}_{2}  +  \dual{\eD^0\sigma,v}_{1} &= \dual{F,v}_1 && (\forall v\in H{M}^1), \label{eq:aha02a} \\
     \dual{u, \eD^0\tau}_1 &= 0   &&(\forall \tau \in H{M}^0). \label{eq:aha02b}
  \end{align}
\end{subequations}
%if and only if it is
%this mixed formulation of $\mathfrak{B}^*$is descrived by the saddle point problem of following degenerate action $\mathcal{S}$.
%  \begin{align}
%    \mathcal{S}(v,\tau) = \frac{1}{2}\dual{\eD^1 v, \eD^1 v}_{k+1} - \dual{v,\eD^0\tau}_{k} - \dual{F,v}_k\notag\\
%    \mathcal{S}(u, \tau) \leq \mathcal{S}(u, \sigma) \leq \mathcal{S}(v,\sigma)
%  \end{align}
\end{theorem}
\begin{proof}
Let $(u,\sigma)\in HM^1\times HM^0$ satisfy \eqref{eq:aha05}.
For any $\epsilon\in\mathbb{R}$ and $(v,\tau)\in HM^1\times HM^0$, we have
\[
    0 \leq \mathcal{S}(u + \epsilon v,\sigma) - \mathcal{S}(u,\sigma) = \frac{\epsilon^2}{2}\dual{\eD^1 v, \eD^1 v}_{2}
 + \epsilon [\dual{\eD^1 u, \eD^1 v}_{2} + \dual{v,\eD^0\tau}_{1} - \dual{F,v}_1].
\]
Therefore, letting $\epsilon \downarrow 0$, we obtain
\[
    \dual{\eD^1 v, \eD^1 v}_{2} + \dual{\eD^0\tau,v}_{1} - \dual{F,v}_1 \leq 0,
\]
and, letting $\epsilon \uparrow 0$,
\[
   \dual{\eD^1 v, \eD^1 v}_{2} + \dual{\eD^0\tau,v}_{1} - \dual{F,v}_1 \geq 0.
\]
Consequently, we deduce \eqref{eq:aha02a}.
Moreover, \eqref{eq:aha02b} is verified using
\[
 0 \geq \mathcal{S}(u, \sigma + \epsilon q) - \mathcal{S}(u,\sigma) = -\epsilon \dual{u, \eD^0\tau}_1.
\]
The converse is proved in the similar way.
\end{proof}
We have shown that the Hodge Laplacian problem can be formulated as a saddle point problem. After a discretization, the saddle point formulation be written as a matrix form,
\begin{equation}
\begin{bmatrix}
    \begin{array}{rr}
      \bm{A} & \bm{B}^T\\
      \bm{B} & \bm{0}
    \end{array}
\end{bmatrix}
\begin{bmatrix}
  \bm{u} \\
  \bm{\sigma} \\
\end{bmatrix}
=
\begin{bmatrix}
  \bm{F}\\
  0 \\
\end{bmatrix}
.
\end{equation}
Since the coefficient matrix is in general indefinite, a special solver is required.
There are many methods for solving matrix equations.
In this study, we employed the Arrorw-Hurwicz Algorithm, a kind of Uzawa-type iterative solution method developed in economics \cite{saad(2003),aha}.
\begin{center}
\textbf{Arrorw-Hurwicz Algorithm (AHA)}
  \begin{enumerate}
      \item Chose an initial guess $\bm{u}^0$ and $\bm{\sigma}^0$.
      \item For k=0,1,2,...., until convergence of iterative error Do:
      \item \ \ \ \ $\bm{\sigma}^{k+1} = \bm{\sigma}^k +\delta \bm{B} \bm{u}^k$
      \item \ \ \ \ $\bm{u}^{k+1} = \bm{u}^k +\omega(\bm{F} - \bm{A}\bm{u}^k - \bm{B}^T \bm{\sigma}^{k+1})$
      \item EndDo
  \end{enumerate}
\end{center}
The $\bm{A}$ matrix is the main part of the Laplacian and contains the scalar Laplacian and the vector Laplacian.
The $\bm{B}$ matrix is the term that represents $\Div A = 0$.
The potential $\phi$ term has degenerated.
Since $\bm{A}$ alone is an indefinite problem, it is an iterative method that oscillates toward the unique saddle point while taking $\Div A=0$ into account.
\section{Numerical examples}
\label{sec:example}
\subsection{Numerical Example 1}
We checked the mesh convergence of our proposed scheme under the problem of an exact solution in supper cubic.
$x\in[0,1],y\in[0,1],z\in[0,1],t \in [0,1]$.
The source data is following equation,
  \begin{gather}
  F(x,y,z,t) = \rho dx_0 + j_x dx_1 + j_y dx_2 + j_z dx_3 \notag \\
  = -  3 \pi^2 \sin (\pi x) \sin (\pi y) \sin (\pi z) \cos(2 \pi  t) dx_0 \notag\\
  +  3 \pi^2 \cos (\pi x) \sin (\pi y) \sin (\pi z) \sin(2 \pi t) dx_1 \notag \\
  +  3 \pi^2 \sin (\pi x) \cos (\pi y) \sin (\pi z) \sin(2 \pi  t) dx_2 \notag \\
  -  6 \pi^2 \sin (\pi x) \sin (\pi y) \cos (\pi z) \sin(2 \pi t) dx_3
  \end{gather}
The exact solution of this problem is bellow.
  \begin{align}
    u(x,y,z,t) = \phi dx_0 + A_x dx_1 + A_y dx_2 +  A_z dx_3\notag \\
     =\sin (\pi x) \sin (\pi y) \sin   (\pi z) \cos(2 \pi t) dx_0 \notag \\
     +\cos (\pi x) \sin (\pi y)\sin (\pi z) \sin(2 \pi t) dx_1 \notag \\
                   +\sin (\pi x) \cos (\pi y) \sin (\pi z) \sin(2 \pi t) dx_2\notag \\
                 -2\sin (\pi x) \sin (\pi y) \cos  (\pi z) \sin(2 \pi t)dx_3
  \end{align}
  This solution satisfied $\Div A=0$, the vector potentials are toward the boundary face's normal direction, and $\phi$ is equal to zero on the boundary face.
  Table \ref{ec} shows the dependency of mesh size of $L^2$ error, and the convergence order is $r_h$.
\begin{equation}
  E_h = \|u - u_h \|_1, r_h = \log(E^{n+1}/E^n) / \log(h_n/h_{n+1})
\end{equation}
\begin{table}[htbp]
\begin{center}
\begin{tabular}{c|c|c|c}
\hline
$N$ & $h$ & $E_h$ & $r_h$ \\
\hline
 \hline
20736  & 0.0833     & 0.15988 & \\
65536  & 0.0625     & 0.09132 & 1.94678\\
160000 & 0.05       & 0.05886 & 1.968276\\
331776 & 0.0041677  & 0.04104 & 1.977904\\
614656 & 0.0035714  & 0.03022 & 1.98531\\
 \hline
\end{tabular}
 \end{center}
\caption{Errors and convergence rates}
  \label{ec}
\end{table}
The convergence rate is visibly linear under the log-log plot.
The numerical example 1 shows optimal second-order convergence.
%\begin{figure}[htbp]\label{ex1}
% \centering
% %\includegraphics[width=0.6\textwidth]{img/error.eps}
%\caption{second order mesh convergence of the example1}
%\end{figure}
\subsection{Numerical Example 2}
We show a practical example for electromagnetic simulation.
Figure \ref{setting:example2} shows the problem setting of the virtual plasma source.
\begin{figure}[htbp]
 \centering
 \includegraphics[width=0.8\textwidth]{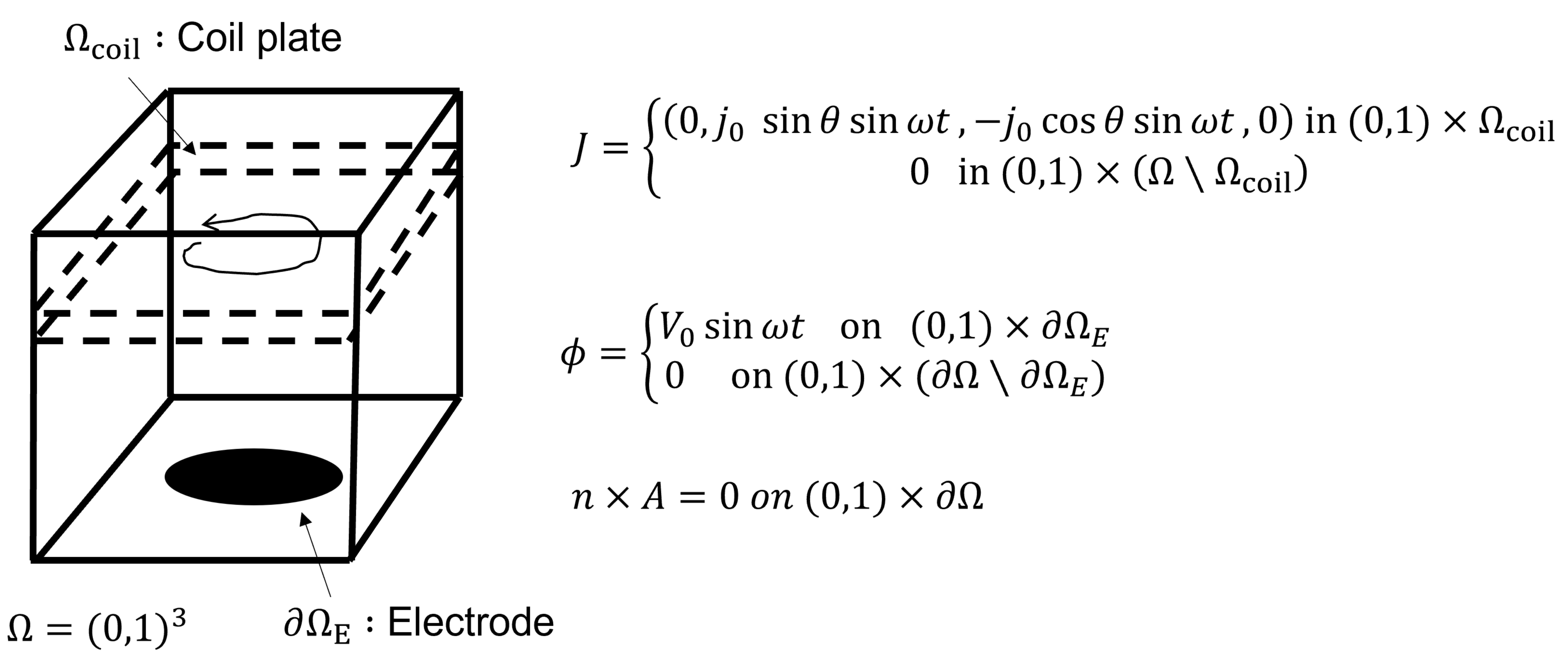}
\caption{problem setting of example 2}
\label{setting:example2}
\end{figure}
The calculation domain is $\Omega=[0,1]^3$ and $\theta = \tan^{-1}(y-0.5/x-0.5)$.
There is no plasma in this case, but it is a shape that could be used as a plasma source for Inductively coupled plasma(ICP) and capacitively coupled plasma(CCP).
The electric current($j_0=1$) flows around the coil plate, set at the height of z=2/3.
The electric charge doesn't exist.
The boundary condition of vector potential is metal boundary condition as $n\times A =0$.
The electrode is set on the bottom and applied the sinusoidal wave($V_0=100$).
The other boundary condition of the scalar is grounded($\phi=0$).
Our theoretical discussion was performed under the zero boundary condition. However, exaple2 is a more complex boundary condition because we can obtain a reasonable solution under the more practical checking problem setting.
Fig.\ref{aaa} and Fig.\ref{bbb} show the calculation results of the scalar and vector potential distribution by the cross-sectional view of x=0 under the coarse and fine mesh.
\begin{figure}[htbp]
 \centering
 \includegraphics[width=0.8\textwidth]{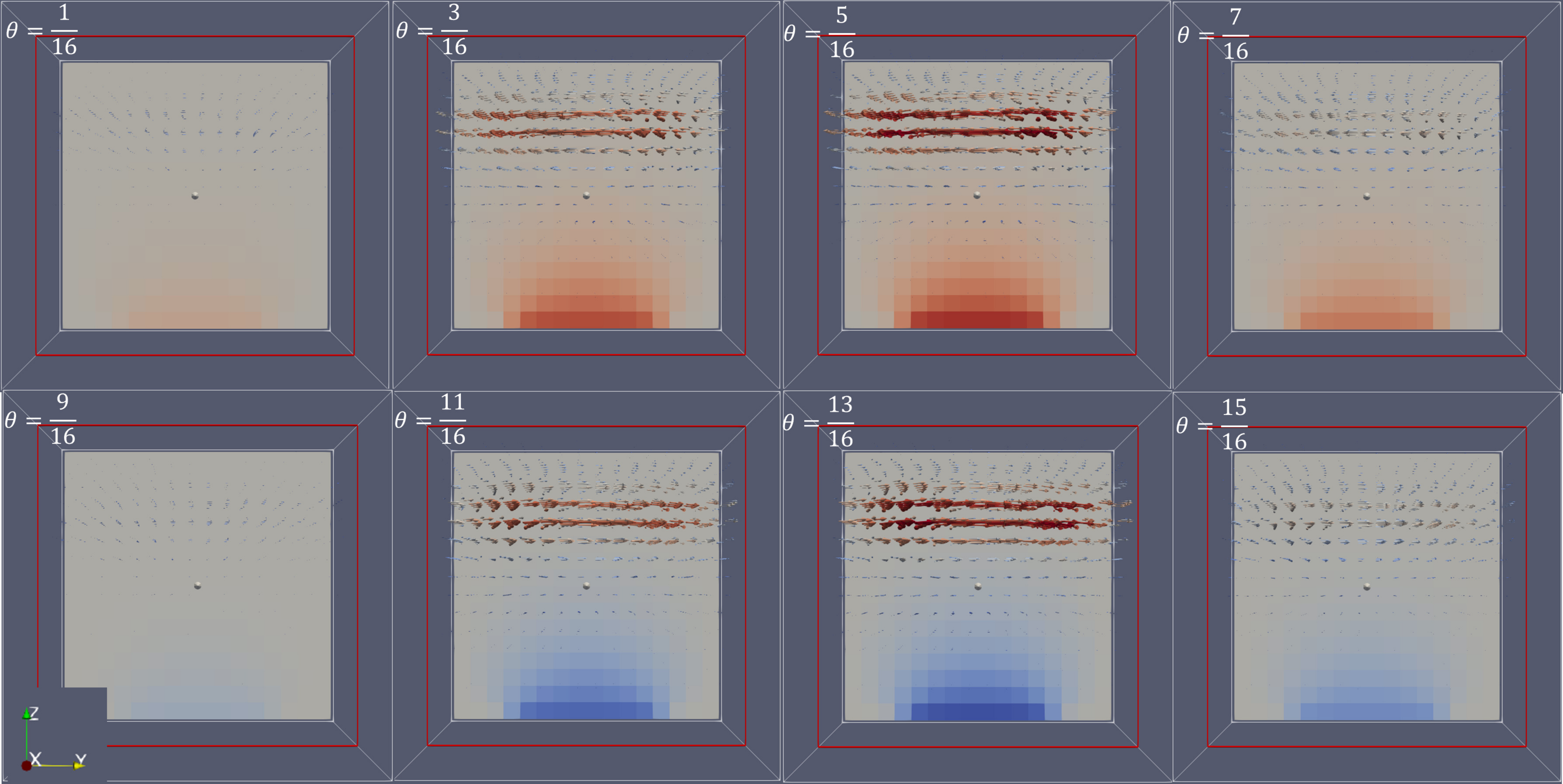}
\caption{cross sectional view (x = 0.5 plane) of the scalar and vector potential with coarse mesh}
  \label{aaa}
\end{figure}
\begin{figure}[htbp]
 \centering
 \includegraphics[width=0.8\textwidth]{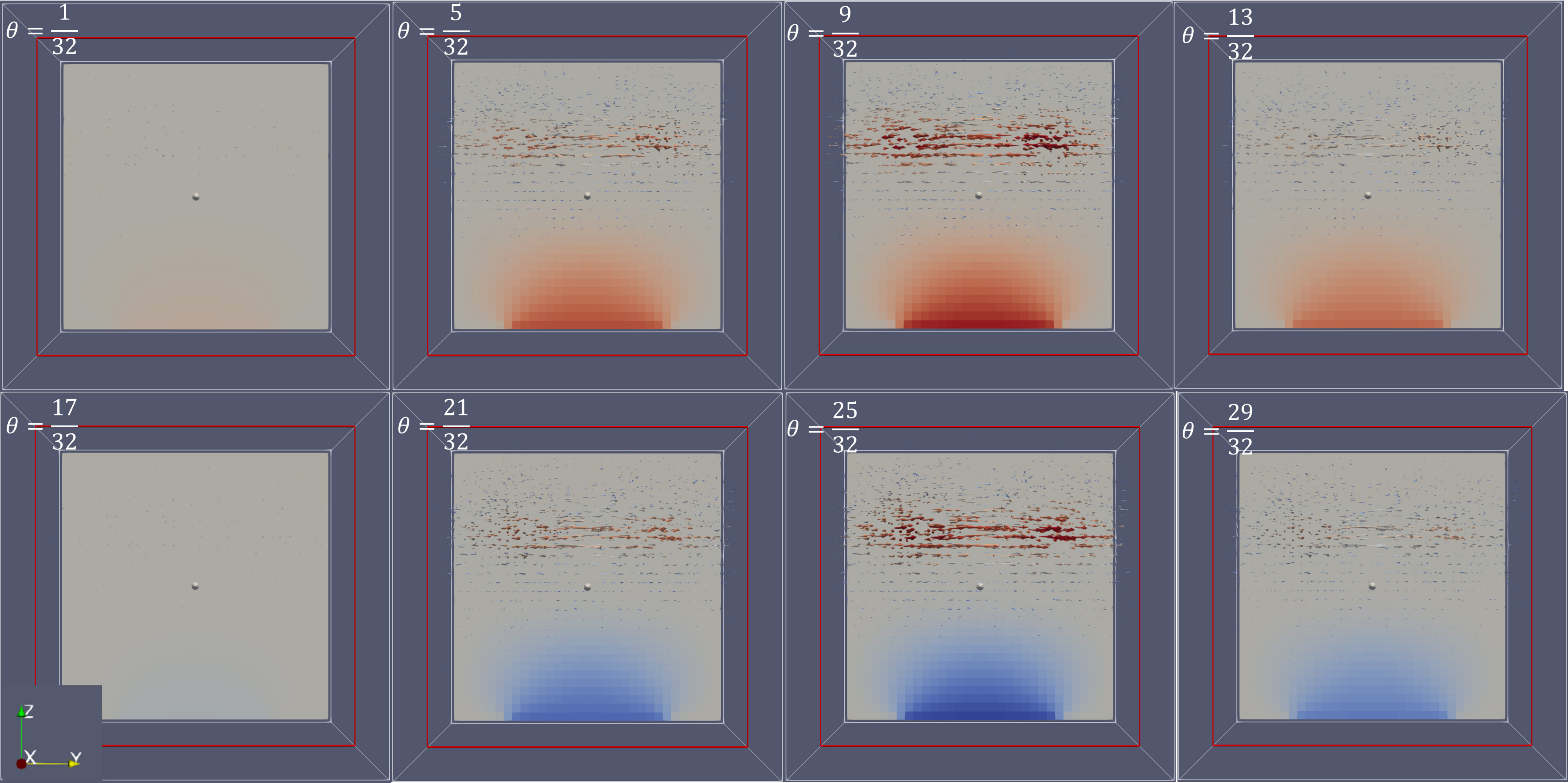}
\caption{cross sectional view (x = 0.5 plane) of the scalar and vector potential with fine mesh}
  \label{bbb}
\end{figure}
The 1st order(r=1) elements of space-time are used.
However, the visualization is performed as the average value of each element, and the scalar and the vector potential are separated.
It can be confirmed that the solution satisfies the discrete maximum principle with the maximum value at the electrode, depending on the change in the sine waveform given to the electrode.
As for the vector potential, a tendency was obtained that it is parallel to the current direction, and its absolute value rapidly decreases as it moves away from the coil plate.
Figures \ref{ccc} and \ref{ddd} show the calculated scalar and vector potentials on the (x=0) cross section under the coarse and fine meshes.
\begin{figure}[htbp]
 \centering
 \includegraphics[width=0.8\textwidth]{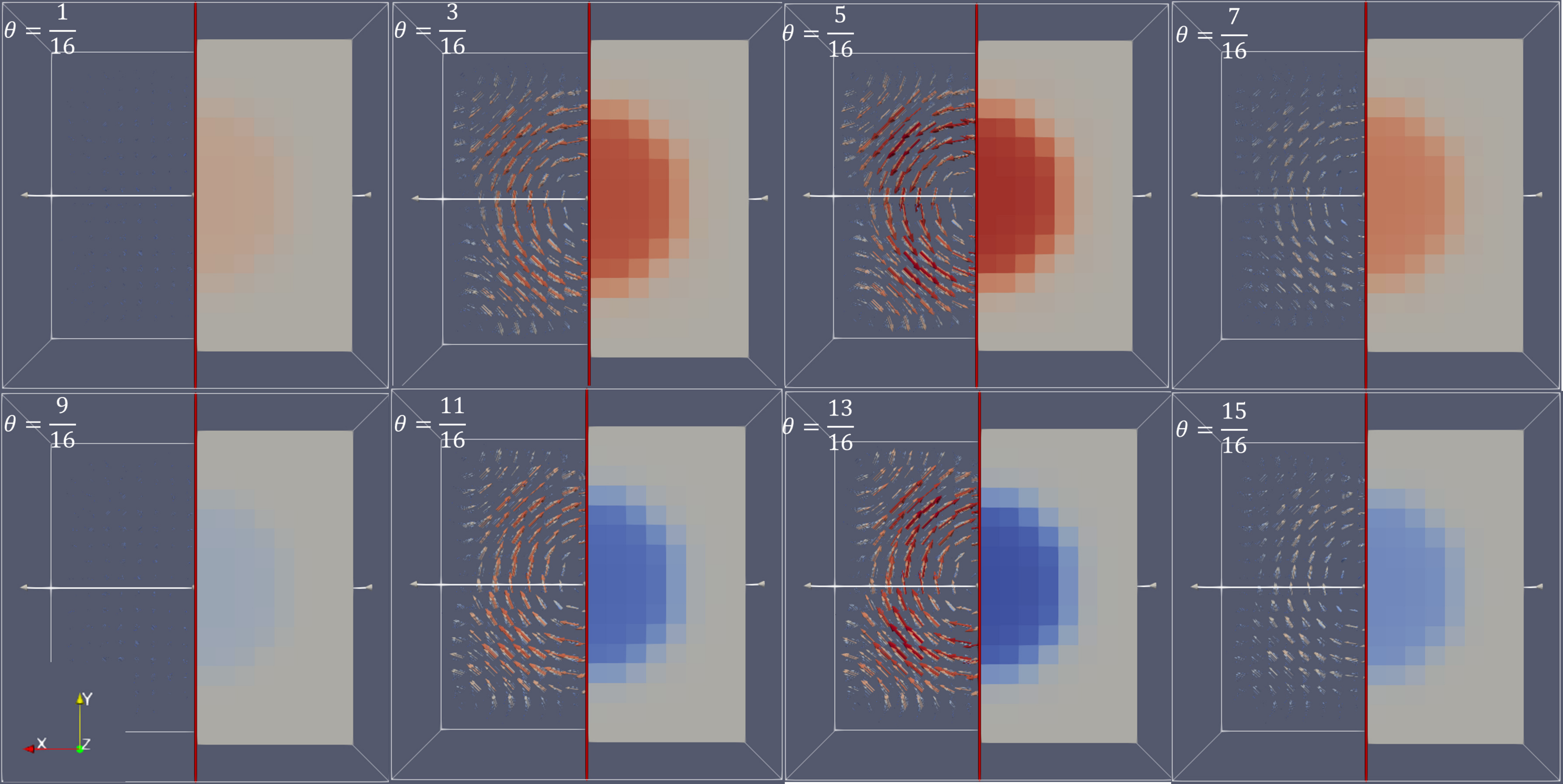}
\caption{cross sectional view (z = 0.5 plane) of the scalar and vector potential with coarse fine mesh}
\label{ccc}
\end{figure}
\begin{figure}[htbp]
 \centering
 \includegraphics[width=0.8\textwidth]{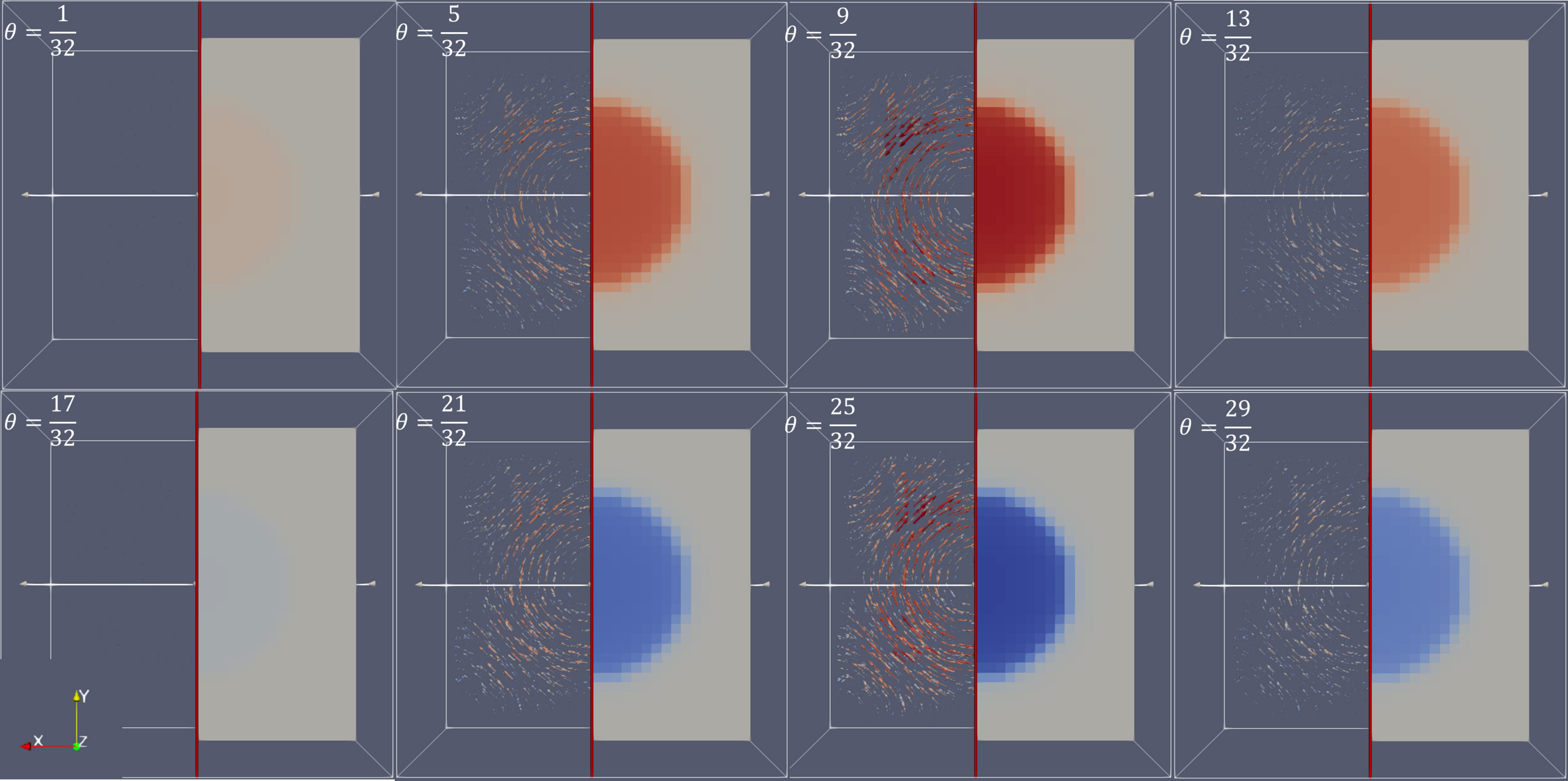}
\caption{cross sectional view (z = 0.5 plane) of the scalar and vector potential with coarse mesh}
\label{ddd}
\end{figure}
The results show that the vector potential is generated in a vortex.
In the boundary condition problem, the intensity of the solution increases from the center to the outside, but since the tangential direction of the boundary is zero and the vector value on the boundary has only a normal component as the metal boundary condition, the absolute value of the intensity decreases toward the normal direction as it moves outward.
\section{Conclulding Remarks}
\label{sec:conclusion}
In this paper, the periodic steady-state of electromagnetic fields is calculated using differential forms in 4-dimensional space-time.
In the conventional method, the scalar potential 0-form and the vector potential 1form are considered separately on the three-dimensional de Rham Complex.
In the proposed formulation, the scalar potential 0-form and the vector potential 1-form are treated simultaneously as 1-form in the 4-dimensional spacetime.
We consider a direct product space by the shifted differential $(k-1)$-form and the differential $k$-form and its Hilbert Complex.
And the proposed formulation is equivalent to the weakly mixed formulation of the Hodge-Laplacian in the  4-dimensional spacetime.
Then, the wellposedness of this 4-dimensional Hodge-Laplacian is shown using the FEEC framework.
For the function space of the discrete solution, we considered the direct product space of shifted $ $(k-1)$ $-form space(cubical element space) in the 3-dimension and k-form space(cubical element space).
It was also shown that this product space coincides with the 4-dimensional cubical element space.
The unbounded cochain map between a 4-dimensional complex and its approximate complex exists by using the approximate map from the cubical element spaces to the differential form space on the 3-dimensional space.
Then, we show the well-posedness of the formulation of the weakly mixed problem using the framework of FEEC theory.

We have tested the proposed method on two example problems.
The example 1, the exact solution exists, the discrete solution converges to the exact solution in optimal order by the proposed method.
The more concrete example2 shows that the proposed method can solve problems with non-zero boundary conditions for the scalar potential.
These results support our theoretical analysis and the usefulness of our proposed method.
In this paper, we have focused on a model in which the time derivative term is neglected.
In the future, it is expected that calculate Maxwell's equations with time terms taken into account and develop another periodic steady problem such as fluid fields
Furthermore, extend to the coupled problems of electromagnetic and fluid fields for plasma simulation.

\section*{Acknowledgement}
I'd like to thank for flutiful disscusion with Dr. Masaru Uchiyama.
\bibliography{mybibfile}
\end{document}